\documentclass[11pt, twoside]{article}
\pdfoutput=1

\usepackage{graphicx}
\usepackage[caption=false]{subfig}
\usepackage{amsmath}
\usepackage{amssymb,amsfonts}
\usepackage{amsthm}
\usepackage{bm}
\usepackage{mathrsfs}
\usepackage{amssymb}
\usepackage{multirow}

\usepackage[margin=1in]{geometry}

\usepackage{algorithm}
\usepackage{algorithmic}

\numberwithin{equation}{section}
\usepackage{multirow}
\usepackage{makecell}
\usepackage{booktabs}
\theoremstyle{definition}
\newtheorem{theorem}{Theorem}

\newtheorem{lemma}{Lemma}
\newtheorem{corollary}{Corollary}
\newtheorem{definition}{Definition}

\usepackage{cite}
\usepackage{hyperref}
\usepackage[nameinlink]{cleveref}

\usepackage{fancyhdr}
\begin{document}
	
	\title{Superconvergent  HDG methods for  Maxwell's equations via
		the $M$-decomposition}

\author{Gang Chen%
	\thanks{ College of  Mathematics, Sichuan University, Chengdu 610064, China (\mbox{cglwdm@uestc.edu.cn}).}
	\and
	Peter Monk%
	\thanks{Department of Mathematics Science, University of Delaware, Newark, DE, USA (\mbox{monk@udel.edu}).}
	\and
	Yangwen Zhang%
	\thanks{Department of Mathematics Science, University of Delaware, Newark, DE, USA (\mbox{ywzhangf@udel.edu}).}
}

\date{\today}

\maketitle

\begin{abstract}
	The concept of  the  $M$-decomposition was introduced by Cockburn et al.\ in Math. Comp.\ vol.\  86 (2017),  pp.\ 1609-1641 {to provide criteria to guarantee optimal convergence rates for the Hybridizable Discontinuous Galerkin (HDG) method for coercive elliptic problems}. In that paper they systematically constructed  superconvergent hybridizable discontinuous Galerkin (HDG) methods to approximate  the solutions of elliptic PDEs on unstructured meshes.  In this paper, we use the $M$-decomposition to construct HDG methods for the Maxwell's equations on unstructured meshes in two dimension. In particular, we show the any choice of spaces having an $M$-decomposition,
	together with sufficiently rich auxiliary spaces, has an optimal error estimate and superconvergence even
	though the problem is not in general coercive. Unlike the elliptic case, we obtain a superconvergent rate for the curl of the solution, not the solution, and this is confirmed by our numerical experiments.  
\end{abstract}

\section{Introduction}
A large number of computational techniques have been developed for solving Maxwell's equations in both the frequency and time domains. In the frequency domain, and in the presence of inhomogeneous penetrable media, the  finite element method is often used. It has an additional advantage compared to finite differences in that it can handle complex geometries.

Methods using $\bm{H}(\text{curl};\Omega)$-conforming edge elements have been widely studied, see for example \cite{Nedelec_mixed_NM_1986,Nedelec_mixed_NM_1980,Hiptmair_electromagnetism_Acta_2002,Monk_Maxwell_Book_2003,Monk_Maxwell_NM_1992,Zhong_Maxwell_JCM_2009}. The implementation of the conforming method, particularly higher order elements, is complicated. Hence,  non-conforming methods provide an  interesting   alternative for this kind of problem that may also be attractive for nonlinear problems.
In particular, Discontinuous Galerkin (DG) methods  have been used to approximate the solution of the Maxwell's equations for a long time. The first DG method for solving
Maxwell's equations with high frequency was analyzed in \cite{Monk_Maxwell_CMAME_2002}. A local discontinuous Galerkin (LDG) scheme was proposed for the time-harmonic Maxwell's equations with low frequency  was studied in \cite{Perugia_LDG_maxwell_Math_Comp_2003} (see also \cite{Houston_Maxwell_SINUM_2004} for this problem using mixed DG methods).  These methods tend to have many more degrees of freedom than conforming methods so it is interesting to consider hybridizable methods.

This paper is concerned with developing a class of methods for Maxwell's equations in 2D.  Obviously Maxwell's
equations are usually studied in three dimensions, but if the domain and data functions are translation invariant in one direction, the full problem can be decoupled into a pair of problems posed in two dimensions.  To see how this is possible, consider
the usual time harmonic Maxwell system for the electric field $\bm E$ (a complex valued vector function):
\[
{\rm{}curl}\,\mu_r^{-1}{\rm{}curl}\,\bm E-\kappa^2\epsilon_r\bm E=\bm F.
\]
Here $\mu_r$ is the relative magnetic permeability, $\kappa>0$ is the wave number, $\epsilon_r$ is the relative electric permittivity which may be complex valued.  In addition
$\bm F=ik\epsilon_0\bm j$,  where $\bm j$ is the given current density and $\epsilon_0$ is the permittivity of vacuum.

If $\epsilon_r$, $\mu_r$ and $\bm F$ are independent of $x_3$, and if we seek a solution 
$\bm E$ that is also independent of $x_3$ we obtain a simpler partial differential equation for $\bm u=(E_1,E_2)^T$ 
given in (\ref{simple1}) below. In addition, a Helmholtz equation is obtained for $E_3$ but is not the subject of this paper (see 
\cite{Griesmaier_Monk_Helmholtz_JSC_2011,Cui_Zhang_Helmholtz_IMAJNA}). To define the problem for $\bm u$ we need some notation. Because we are now working in two dimensions the curl operator can be defined in two ways depending on whether its argument is a scalar or a vector. We therefore introduce  the following standard definitions where $\bm v$ is a smooth vector function and $p$ is a smooth scalar function:
\begin{align}
\nabla\times\bm v=(-\partial_y,\partial_x)^T\cdot\bm v, \qquad  \nabla\times p= (\partial_y,-\partial_x)^Tp,
\label{curl2D}
\end{align}
Similarly there are two definitions for the cross product again depending on the use of scalar or vector functions. If $\bm n$ is a unit vector (in practice the normal vector to a domain in $\mathbb{R}^2$), we define
\begin{align*}
\bm n\times\bm v =(-n_2,n_1)^T\cdot\bm v,\qquad  \bm n\times p=(n_2,-n_1)^Tp,
\end{align*} 
We can now state the problem we shall study. Let $\Omega$ be a bounded simply-connected Lipschitz polygon in $\mathbb{R}^2$ with connected boundary $\partial\Omega$. Then a typical model problem for $\bm u$ is to seek solutions of the following interior problem:
\begin{subequations}\label{Maxwell_equation_ori_form}
	\begin{align}
	\nabla\times(\mu_r^{-1}\nabla\times \bm{u}) -\kappa^2\epsilon_r\bm u&=\bm{f}&\text{ in }\Omega, \label{simple1}\\
	\bm{n}\times \bm{u} &={g}&\text{ on }\partial\Omega,
	\end{align}
\end{subequations}
where the right hand side is $\bm f=(F_1,F_2)^T$. 
To ensure the uniqueness of the solution to this problem (and hence existence via the Fredholm alternative), we assume that $\mu_r$ is real values and positive.  In addition, either $\Im(\epsilon_r)>0$, or $\Im(\epsilon_r)=0$ and $\kappa^2$ is not a Maxwell eigenvalue, where $\Im(\epsilon_r)$ denotes the image part of $\epsilon_r$.  Note that using the vector form of the problem has been advocated for
example in \cite{brenner0,brenner1} and these papers motivate in part the current study.

In this paper we shall study hybridizable discontinuous Galerkin (HDG) methods applied to Maxwell's equations (\ref{Maxwell_equation_ori_form}). HDG methods for elliptic problems were first proposed in 2009 in  \cite{Cockburn_Gopalakrishnan_Lazarov_Unify_SINUM_2009} and
an analysis using special projections was developed in \cite{Cockburn_Gopalakrishnan_Sayas_Porjection_MathComp_2010}. 
HDG methods have several distinct advantages including: allowing static condensation and hence less global degree of freedoms, flexibility in meshing (inherited from DG methods), ease of design and implementation, and  local conservation of physical quantities.  As a result, HDG methods have been proposed for a large number of problems, see, e.g., \cite{ChenHuShenSinglerZhangZheng_HDG_Convection_Dirtributed_Control_JCAM_2018,HuShenSinglerZhangZheng_HDG_Dirichlet_control3,HuMateosSinglerZhangZhang2,Cesmelioglu_NS_MathComp_2017,Muralikrishnan_Sriramkrishnan_Tran_Bui_iHDG_SISC_2017,Tan_Shallow_Water_SISC_2016,Sheldon_Miller_Jonathan_Fluid_JCP_2016}.

An important property of HDG methods is  the superconvergence of some quantities on  unstructured meshes (after element by element post-processing). One way
to guarantee the existence of an HDG projection and superconvergence is to ensure that the particular discretization 
spaces used in the HDG method satisfy an $M$-decomposition~\cite{Cockburn_M_decomposition_Part1_Math_Comp_2017}.  This reduces the problem of
determining whether a choice of spaces will have good convergence properties to simply checking some
inclusions and evaluating an index (see equation~(\ref{index_of_I_M})). This method of analysis has been extended to other applications, see for example~\cite{Cockburn_M_decomposition_H1_SINUM_2018,Cockburn_M_decomposition_Elasticity_IMA_2018,Cockburn_M_decomposition_Stokes_IMA_2017,Cockburn_M_decomposition_Part3_M2AN_2017,Cockburn_M_decomposition_Part2_M2AN_2017}.

The HDG method has been applied to Maxwell's equations in \cite{Cockburn_maxwell_HDG_JCP_2011} but without an error analysis. Later on, an error analysis was provided 
in \cite{Chen_maxwell_HDG_2017,Chen_Maxwell_HDG_CMAME_2018}  for  zero frequency and  in \cite{Feng_Maxwell_CMAM_2016,Lu_hp_HDG_Maxwell_Math_Comp_2017} for impedance boundary conditions and high wave number.  These papers did not use the $M$-decomposition and only considered simplicial elements.

The aim of this paper is to extend the concept of the $M$-decomposition to  time-harmonic Maxwell's equations with Dirichlet boundary condition in 2D. The main novelty of our paper is that we show that provided the HDG spaces satisfy the conditions for an $M$-decomposition, and certain auxiliary spaces contain constant piecewise polynomials, an optimal error estimate will hold as well as a super-convergence of the curl of the field (as was
observed in \cite{Cockburn_maxwell_HDG_JCP_2011}).  Note that in our context superconvergence of the curl of the field is important because this implies that both the electric and magnetic fields can be approximated at the same rate. We then use the $M$-decomposition to exhibit
finite element spaces with optimal convergence on triangles, parallelograms and squares.  Our convergence theory is supported by numerical examples in each case.

The outline of the paper is as follows.  In \Cref{HDG_mathods}, we set some notation and give the HDG formulation of \eqref{Maxwell_equation_ori_form}. In \Cref{M_decompositions_for_Maxwell_equations}, we follow the seminal work \cite{Cockburn_M_decomposition_Part1_Math_Comp_2017} to introduce the concept of the $M$-decomposition for Maxwell's equation. The error analysis is given in \Cref{Error_estimates}, we obtain optimal convergence rate for the electric field $\bm u$ and superconvergence rate for $\nabla \times \bm u$.
The construction of example spaces and  numerical experiments are provided to confirm our theoretical results in \Cref{construction_and_numerical_experiments}.  We end with a conclusion.

\section{The HDG method}
\label{HDG_mathods}

We start by defining some notation. For any sufficiently smooth  bounded domain $\Lambda\subset\mathbb{R}^2$, let $H^{m}(\Lambda)$ denote the usual  $m^{th}$-order Sobolev space of scalar functions on $\Lambda$, and $\|\cdot\|_{m, \Lambda}$, $|\cdot|_{m,\Lambda}$  denote the corresponding norm and semi-norm.
We use $(\cdot,\cdot){_\Lambda}$ to denote the complex inner product on $L^2(\Lambda)$.
Similarly, for the boundary $\partial \Lambda$ of $\Lambda$, we use $\langle\cdot,\cdot\rangle_{\partial\Lambda}$ to denote the $L^2$ inner product.
Note that bold face fonts will be used for vector analogues of the Sobolev spaces along with vector-valued functions. 

Recalling the definition of the curl operators in 2D in (\ref{curl2D}), for $\Lambda\subset\mathbb{R}^2$ we next define
\begin{align*}
\bm{H}(\text{curl};\Lambda):=\{\bm u\in \bm L^2(\Lambda): \nabla \times \bm u \in L^2({\Lambda}) \},  \quad  \bm H_0({\rm curl;\Lambda}) =\{\bm u\in \bm  H({\rm curl;\Lambda}): \bm n \times \bm u=0 \textup{ on } \partial \Lambda\}\\
{H}(\text{curl};\Lambda):=\{ u\in  L^2(\Lambda): \nabla \times u \in \bm L^2(\Lambda) \},\quad H_0({\rm curl;\Lambda}) =\{ u\in   H({\rm curl;\Lambda}): \bm n \times  u=\bm 0 \textup{ on } \partial \Lambda\},\\
\bm H({\rm div}(\Lambda):=\{\bm u\in \bm L^2(\Lambda): \nabla \cdot \bm u \in L^2({\Lambda}) \},  \quad  \bm H({\rm div}^0;\Lambda)=\{\bm u\in \bm H({\rm div};\Lambda): \nabla\cdot\bm u=0\}.
\end{align*}
where $ \bm n$ is  the  unit outward normal vector on $\partial \Lambda$.

Let $\mathcal{T}_h:=\{K\}$ denote a conforming mesh of $\Omega$, where $K$ is a Lipschitz polygonal elemen
with finitely many edges. For each $K\in\mathcal{T}_h$, we let $h_K$ be the infimum of the diameters of circles containing $K$ and denote the mesh size $h:=\max_{K\in\mathcal{T}_h}h_K$. {We shall need more assumptions on the mesh to perform our analysis. These assumptions replace the usual ``shape regularity'' assumption but we delay a discussion of this point until Section~\ref{messy_discussion}.} Let
$\partial\mathcal T_h$ denote the set of edges $F\subset \partial K$ of the elements $K\in\mathcal{T}_h$ (i.e. edges of distinct triangles are counted separately) and let $\mathcal{E}_h$ denote the set of edges in the mesh $\mathcal T_h$.   
We denote by $h_F$ the length of the edge $F$.
We abuse notation by using $\nabla \times$, $\nabla\cdot$ and $\nabla$ for broken \text{curl}, \text{div} and gradient operators with respect to mesh partition $\mathcal{T}_h$, respectively.
To simplify the notation, we also define a function ${\bf h}$ on $\mathcal T_h$, $\partial\mathcal T_h$ and $\mathcal E_h$ which dependenting on circumstances:
\begin{align*}
{\bf h}|_K=h_K,\quad\forall K\in \mathcal T_h,\qquad
{\bf h}|_{\partial K}=h_K,\quad\forall K\in \mathcal T_h,\qquad
{\bf h}|_{F}=h_F,\quad\forall F\in \mathcal E_h.
\end{align*}

For $u, v\in L^2(\mathcal{T}_h)$ and $\rho, \theta \in L^2(\partial\mathcal T_h)$, we define the following  inner product and norm
\begin{align*}
(u,v)_{\mathcal{T}_h} =\sum_{K\in\mathcal{T}_h}(u,v)_K,  \quad \|v\|^2_{\mathcal{T}_h}=\sum_{K\in\mathcal{T}_h}\|v\|^2_K,\quad 
\langle \rho, \theta \rangle_{\partial\mathcal{T}_h}=\sum_{K\in\mathcal{T}_h}\langle \rho,     
\theta\rangle_{\partial K}, \quad \|\theta\|^2_{\partial\mathcal{T}_h}=\sum_{K\in\mathcal{T}_h}\|\theta\|^2_{\partial K}.
\end{align*}

Given a choice of three finite dimensional polynomial spaces $V(K)\subset H^1(K)$, $\bm W(K)\subset \bm H(\text{curl}; K)$ and $ \bm M(F)\subset \bm L^2(F)$,  where $K$ is an arbitrary element in the mesh and $F$ is an arbitrary edge, we define the global spaces by
\begin{align*}
V_h&:=\{ v\in L^2(\mathcal{T}_h): w|_K\in V(K), K\in\mathcal{T}_h\},\\
\bm W_h&:=\{ \bm w\in  \bm L^2(\mathcal{T}_h): \bm w|_K\in \bm W(K), K\in\mathcal{T}_h\},\\
\bm M_h&:=\{\bm \mu\in\bm L^2(\mathcal E_h): \bm \mu|_F\in\bm M(F), F\in\mathcal E_h\}.
\end{align*}
{For later use, for any non-negative integer $k$, let $\mathcal{P}_k(K)$ denote the standard  space of polynomials in two variables have total degree less than or equal to $k$.}

Next, to give the HDG fomulation of \eqref{Maxwell_equation_ori_form}, we need to rewrite it into a mixed form. Let $ q=\mu^{-1}\nabla\times\bm u$ in \eqref{Maxwell_equation_ori_form} to get the following mixed form
\begin{subequations}\label{Maxwell_equation_mixed_form}
	\begin{align}
	\mu_r q-\nabla\times\bm u&=\bm 0&\text{ in }\Omega,\label{Mixed-1}\\
	\nabla\times q -\kappa^2\epsilon_r\bm u&=\bm{f}&\text{ in }\Omega,\label{Mixed-2}\\
	\bm{n}\times \bm{u} &={g}&\text{ on }\partial\Omega.\label{Mixed-3}
	\end{align}
\end{subequations}
As usual for HDG, the upcoming method and analysis are based on the above mixed form.

For the convenience, we next give the following integration by parts formula for each curl operator in two-dimensions. The proof is  followed by a standard density argument and hence we omit it here.
\begin{lemma}\label{integration_by_parts}  Let $K$ be an element in the mesh $\mathcal T_h$, and let
	$\bm u\in \bm H(\text{curl};K)$ and $r\in H(\text{curl};K)$. Then we have
	\begin{subequations}\label{parts}
		\begin{align}
		(\nabla\times\bm u,r)_K=\langle\bm n\times\bm u,r\rangle_{\partial K}+(\bm u,\nabla\times r)_K,\label{integration_by_parts1}\\	
		(\nabla\times r,\bm u)_K=\langle\bm n\times r, \bm u\rangle_{\partial K}+(r,\nabla\times \bm u)_K,\label{integration_by_parts2}
		\end{align}
	\end{subequations}
	where $\bm n$ is the unit outward normal to $K$.
\end{lemma}

We can now derive the HDG method for (\ref{Maxwell_equation_mixed_form}) by multiplying each equation by the appropriate discrete test
function, integrating element by element and use integration by parts  (see \eqref{parts}) element by element in the usual way (c.f.~\cite{Cockburn_Gopalakrishnan_Lazarov_Unify_SINUM_2009}).
Summing the results over all elements, the  HDG methods seeks an approximation to $(q,\bm u, \bm u|_{\mathcal{E}_h})$, by $(q_h,\bm u_h, \widehat{\bm u}_h)\in V_h\times\bm W_h\times  \bm M_h$, such that
\begin{subequations}\label{Maxwell_equation_HDG_form_ori}
	\begin{align}
	(\mu_r q_h, r_h)_{\mathcal{T}_h}-(\bm u_h,\nabla\times r_h)_{\mathcal{T}_h}-\langle \bm n\times\widehat{\bm u}_h, r_h \rangle_{\partial\mathcal{T}_h}&=0,\label{Maxwell_equation_HDG_form_ori_1}\\
	( q_h,\nabla\times\bm v_h)_{\mathcal{T}_h}
	+\langle\bm n\times\widehat{ q}_h,{\bm v}_h \rangle_{\partial\mathcal{T}_h}-(\kappa^2\epsilon_r\bm u_h,\bm v_h)_{\mathcal{T}_h}&=(\bm f,\bm v_h)_{\mathcal{T}_h},\label{Maxwell_equation_HDG_form_ori_2}\\
	\langle\bm n\times\widehat{ q}_h,\widehat{\bm v}_h \rangle_{\mathcal{F}_h/\partial\Omega}&=0,\label{Maxwell_equation_HDG_form_ori_3}\\
	\langle\bm n\times\widehat{\bm u}_h,\bm n\times\widehat{\bm v}_h \rangle_{\partial\Omega}&=\langle
	g,\bm n\times\widehat{\bm v}_h \rangle_{\partial\Omega}\label{Maxwell_equation_HDG_form_ori_4}
	\end{align}
	for all $(r_h,\bm v_h, \widehat{\bm v}_h)\in V_h\times\bm W_h\times \bm M_h$,  and the choice of $\bm n\times\widehat q_h$ follows
	the usual HDG pattern,
	\begin{align}
	\bm n\times\widehat{ q}_h=
	\bm n\times q_h+\tau\bm n\times(\bm u_h-\widehat{\bm u}_h      )\times\bm n,\label{Maxwell_equation_HDG_form_ori_5}
	\end{align}
\end{subequations}
where $\tau$ is a penalization parameter taken to be positive and piecewise constant on the edges of the mesh (more details will be given later).


\section{$M$-decompositions}
\label{M_decompositions_for_Maxwell_equations}

In this section, we follow   the seminal paper \cite{Cockburn_M_decomposition_Part1_Math_Comp_2017} to give the concept of the $M$-decomposition for  Maxwell's equation in two dimensions. 
To do this, we need an appropriate combined trace operator ${\rm tr}: V(K)\times \bm W(K)\mapsto L^2(\partial K)$ defined
as follows:
\begin{align}\label{combine_trace_operator}
{\rm tr}(v,\bm w):=(\bm n\times v+\bm n\times\bm w\times\bm n)|_{\partial K}.
\end{align}

\begin{definition}\label{def_M_decomposition}
	We say that $ V(K)\times\bm W(K)$ admits an $ M$-decomposition when the following conditions are met:
	\begin{subequations}
		\begin{align}\label{M_decomposition_1}
		\bm n\times V(K)\subset \bm M(\partial K),\qquad
		\bm n\times\bm W(K)\times\bm n\subset \bm M(\partial K),
		\end{align}
		and there exists a subspace $\widetilde{V}(K)\times \widetilde{\bm W}(K)$ of $V(K)\times {\bm W}(K)$ satisfying
		\begin{align}
		&\nabla\times V(K)\subset\widetilde{\bm W}(K),\qquad
		\nabla\times\bm W(K)\subset\widetilde{ V}(K),\label{M_decomposition_2}\\
		&{\rm tr}:\left( \widetilde{ V}^{\perp}(K)\times\widetilde{\bm W}^{\perp}(K)\right) \to \bm M(\partial K) \text{ is an isomorphism},\label{M_decomposition_3}\\
		&\text{for any }\bm{\mu}\in\bm M(\partial K),\text{ if }\bm n\times\bm{\mu}=0,\text{ it holds }\bm{\mu}=\bm 0.\label{M_decomposition_4}
		\end{align}
	\end{subequations}
	Here $\widetilde{V}^{\perp}(K)$ and $\widetilde{\bm W}^{\perp}(K)$ are the $L^2(K)$-orthogonal complements of $\widetilde{ V}(K)$ in $ V(K)$, and $\widetilde{\bm W}(K)$ in $\bm W(K)$, respectively.\end{definition}

We notice that conditions \eqref{M_decomposition_1}, \eqref{M_decomposition_2} and \eqref{M_decomposition_3}  follow closely those in the $M$-decomposition in \cite{Cockburn_M_decomposition_Part1_Math_Comp_2017}. Condition \eqref{M_decomposition_4} is  to ensure the uniqueness of $\widehat{\bm u}_h$ determined by the HDG scheme \eqref{Maxwell_equation_HDG_form_ori}. We shall show that  this implies optimal convergence of the associated HDG scheme (under some extra conditions on $\widetilde{V}(K)$ and $\widetilde{\bm W}(K)$) and verify that several families of elements
satisfy the $M$-decomposition. To give some idea of the form such elements can take, we refer to \Cref{table_p0} for examples of $M$-decompositions for Maxwell's equations in two-dimensions when 
\begin{equation}
\bm{M}(\partial K)=\{\bm \mu: \bm\mu|_F= \bm n\times p_0\text{ for some } p_0\in\mathcal{P}_0(F)\mbox{ and for each edge }F\mbox{ of }K\}\label{simpleM}.
\end{equation}

\begin{table}[H]\label{examples}
	\tiny
	\caption{\label{table_p0}Some examples of $M$-decompositions when $M(\partial K)$ is given 
		by (\ref{simpleM}).}	\centering
	
	\begin{tabular}{l|l|l|l|l|l|l}
		\Xhline{1pt}
		
		$K$&$V(K)$&$\bm W(K)$  &$\widetilde{V}(K)$ &$\widetilde{\bm W}(K)$  &$\widetilde{V}^{\perp}(K)$&$\widetilde{\bm W}^{\perp}(K)$  \\
		\hline
		Triangle&$\mathcal{P}_0(K)$	&$\bm{\mathcal{P}}_0(K)$  &$\{0\}$	&$\{\bm 0\}	$   &$\mathcal{P}_0(K)$	&$\bm{\mathcal{P}}_0(K)$\\

		Triangle&$\mathcal{P}_0(K)$	&$\bm{\mathcal{P}}_0(K)\oplus\text{span}\left\{\left(^y_{-x}\right)\right\}$	    &$\mathcal{P}_0(K)$	&$\{\bm 0\}	$   &$\{0\}$	&$\bm{\mathcal{P}}_0(K)\oplus\text{span}\left\{\left(^y_{-x}\right)\right\}$\\

		Square&$\mathcal{P}_0(K)$	&$\bm{\mathcal{P}}_0(K)\oplus\text{span}\left\{\left(^y_x\right)\right\}$	    &$\{0\}$	&$\{\bm 0\}	$   &$\mathcal{P}_0(K)$	&$\bm{\mathcal{P}}_0(K)\oplus\text{span}\left\{\left(^y_x\right)\right\}$\\

		Square&$\mathcal{P}_0(K)$	&$\bm{\mathcal{P}}_0(K)\oplus\text{span}\left\{\left(^y_x\right),\left(^y_{-x}\right)\right\}$	    &$\mathcal{P}_0(K)$	&$\{\bm 0\}	$   &$\{0\}$	&$\bm{\mathcal{P}}_0(K)\oplus\text{span}\left\{\left(^y_x\right),\left(^y_{-x}\right)\right\}$\\

		\Xhline{1pt}
	\end{tabular}
\end{table}

To verify that a  given space  $V(K)\times \bm W(K)$  admits an $M$-decomposition, we need to construct the associated spaces  $\widetilde{V}(K)$ and $\widetilde{\bm W}(K)$ in Definition~\ref{def_M_decomposition}. However, 
this is difficult, hence we need  a simple way to verify a given space  $V(K)\times \bm W(K)$ admits an $M$-decomposition. Moreover, if the given space $V(K)\times \bm W(K)$ does not  admit an $M$-decomposition, we need to understand how to find  to build a new space from $V(K)\times \bm W(K)$  that admits an $M$-decomposition.
Following \cite{Cockburn_M_decomposition_Part1_Math_Comp_2017}, we define the $M$-index as follows:
\begin{align}\label{index_of_I_M}
\begin{split}
I_M(V(K)\times\bm W(K)):&=\dim \bm M(\partial K)-\dim\{\bm n\times v|_{\partial K}:v\in V(K),\nabla\times v=\bm 0\}\\
&-\dim\{\bm n\times\bm w\times\bm n|_{\partial K}:\bm w\in\bm W(K),\nabla\times\bm w=0\}.
\end{split}
\end{align}

Now, we state the main result in this section, the proof is found in \Cref{Characterization_M_decompositions}.
\begin{theorem} \label{index_IM_M_decomposition}
	The spaces $V(K)$ and $\bm W(K)$ admit an $M$-decomposition if and only if 
	\begin{subequations}
		\begin{align}
		&\bm n\times V(K)\subset \bm M(\partial K),\qquad \bm n\times\bm W(K)\times\bm n\subset \bm M(\partial K),\label{iM1}\\
		&\nabla\times V(K)\subset\bm W(K),\qquad \nabla\times\bm W(K)\subset{ V}(K),\label{iM2}\\
		&I_M(V(K)\times\bm W(K))=0,\label{iM3}\\
		&\text{For any }\bm{\mu}\in\bm M(\partial K),\text{ if }\bm n\times\bm{\mu}=0,\text{ it holds }\bm{\mu}=\bm 0.\label{iM4}
		\end{align}
	\end{subequations}
\end{theorem}

\Cref{index_IM_M_decomposition} provides a simply way to check if any given choice of spaces $V(K)\times \bm W(K)$ admits an $M$-decomposition by just verifying some inclusions and by calculating  a single number, namely, $I_M(V(K)\times\bm W(K))$. 
Of course the associated spaces $\widetilde{ V}(K)$ and $\widetilde{\bm W}(K)$ are essential to define an HDG  projection  for the a priori error analysis of the method and can be found once $V(K)$ and $\bm W(K)$ are known.

Moreover, the conditions in \Cref{index_IM_M_decomposition} are  ``if and only if'', which means that if $I_M(V(K)\times\bm W(K))$ is not zero,  we need to add to $\bm W(K)$ a space $\delta {\bm W}$ of dimension $I_M(V(K)\times\bm W(K))$ to obtain a new space admitting an $M$-decomposition.

\subsection{Properties of the $M$-decomposition}


We now prove a sequence of lemmas that culminate in the proof of Theorem~\ref{index_IM_M_decomposition}.\begin{lemma}\label{the_triple_orthogonal_property}
	Let $ V(K)\times\bm W(K)$ admit an $M$-decomposition with associated spaces $\widetilde{ V}(K)$ and $\widetilde{\bm W}(K)$. Then we have the following orthogonality property:
	\begin{align}\label{the_triple_orthogonal_property_3}
	\bm M(\partial K)=\gamma\widetilde{ V}^{\perp}(K)\oplus \gamma\widetilde{\bm W}^{\perp}(K),
	\end{align}
	where $\gamma\widetilde{ V}^{\perp}(K):=\{\bm n\times v^{\perp}|_{\partial K}: v^{\perp}\in\widetilde{ V}^{\perp}(K) \}$ and 
	$\gamma\widetilde{\bm W}^{\perp}(K):=\{\bm n\times\bm w^{\perp}\times\bm n|_{\partial K}:\bm w^{\perp}\in\widetilde{\bm W}^{\perp}(K) \}$.
\end{lemma}
\begin{proof} 
	By  condition \eqref{M_decomposition_3} and the definition of the combined  trace operator ${\rm tr}$ in \eqref{combine_trace_operator}, we have 
	\begin{align*}
	\bm M(\partial K)=\gamma\widetilde{ V}^{\perp}(K) +  \gamma\widetilde{\bm W}^{\perp}(K).
	\end{align*}
	Hence, we only need to show that the sum is $L^2(\partial K)$-orthogonal. 
	
	For all $ v^{\perp}\in\widetilde{ V}^{\perp}(K)$, $\bm w^{\perp}\in\widetilde{\bm W}^{\perp}(K)$. On each edge $F$ of $\partial K$, it holds
	\begin{align}\label{decomposition1}
	\bm w^{\perp}=\bm n\times\bm w^{\perp}\times\bm n+(\bm w^{\perp}\cdot\bm n)\bm n.
	\end{align}
	Using  equations \eqref{decomposition1} and  \eqref{integration_by_parts2} we get
	\begin{align}\label{orth_pro_1}
	\begin{split}
	\langle\bm n\times v^{\perp},\bm n\times\bm w^{\perp}\times\bm n    \rangle_{\partial K}&=
	\langle\bm n\times v^{\perp},\bm w^{\perp}-(\bm w^{\perp}\cdot\bm n)\bm n   \rangle_{\partial K}\\
	&= \langle\bm n\times v^{\perp},\bm w^{\perp}   \rangle_{\partial K}\\
	&=(\nabla\times v^{\perp},\bm w^{\perp})_K- ( v^{\perp},\nabla\times\bm w^{\perp})_K.
	\end{split}
	\end{align}
	By equation \eqref{M_decomposition_2}, we have $\nabla\times v^{\perp} \in \widetilde{\bm W}(K)$ and $\nabla\times\bm w^{\perp}\in \widetilde V(K)$, hence 
	\begin{align}\label{the_triple_orthogonal_property_proof_1}
	\langle\bm n\times v^{\perp},\bm n\times\bm w^{\perp}\times\bm n    \rangle_{\partial K}=0.
	\end{align}
	This finishes the proof.
\end{proof}

\begin{lemma}\label{uniquness_of_widetilde_V}
	[Uniqueness of $\widetilde{V}(K)$] If $ V(K)\times\bm W(K)$ admits an $M$-decomposition with associated spaces $\widetilde{ V}(K)$ and $\widetilde{\bm W}(K)$, then the subspace $\widetilde V(K)$ is unique. Moreover,
	\begin{align}\label{uniquness_of_widetilde_V_1}
	\widetilde{V}(K)=\nabla\times\bm W(K).
	\end{align}
\end{lemma}
\begin{proof} 
	By \eqref{M_decomposition_2} we have $\nabla\times\bm W(K)\subset\widetilde{V}(K)$, therefore, we only have to prove 
	$[\nabla\times\bm W(K)]^{\perp}\cap \widetilde{V}(K)=\{0\}$. We take $v\in [\nabla\times\bm W(K)]^{\perp}\cap \widetilde{V}(K)$, then it  satisfies
	\begin{align}\label{uniqueness_1}
	(v,\nabla\times\bm w)_K=0\text{ for all }\bm w\in \bm W(K).
	\end{align}
	By the argument in \eqref{orth_pro_1} and \eqref{uniqueness_1} we have 
	\begin{align}\label{proof_orgonal_1}
	\langle\bm n\times v,\bm n\times\bm w^{\perp}\times\bm n \rangle_{\partial K}=0\text{ for all }\bm w^{\perp}\in \widetilde{\bm W}^{\perp}(K).
	\end{align} 
	For	$\bm n\times v\in \bm M(\partial K)$, since $ V(K)\times\bm W(K)$ admits an $M$-decomposition
	with associated spaces $\widetilde{ V} (K)$ and $ \widetilde{\bm W}(K)$, then by equation \eqref{the_triple_orthogonal_property_3}, 
	there exist $v^{\perp}\in \widetilde V^{\perp}(K)$ and  $\bm w^{\perp}\in\widetilde{\bm W}^{\perp}(K)$ such that
	\begin{align*}
	&\bm n\times v=-\bm n\times v^{\perp}+\bm n\times\bm w^{\perp}\times\bm n,\\
	&\langle\bm n\times v^{\perp},\bm n\times\bm w^{\perp}\times\bm n  \rangle_{\partial K}=0.
	\end{align*}
	By equation \eqref{proof_orgonal_1}, we have
	\begin{align*}
	\|\bm n\times (v+v^{\perp})\|^2_{\partial K}=\langle\bm n\times (v+v^{\perp}),\bm n\times {\bm w}^{\perp}\times\bm n\rangle_{\partial K}=0.
	\end{align*}
	Thus, on $\partial K$, we have $\bm n\times (v+v^{\perp})=\bm 0$, therefore
	\begin{align}\label{sum_equa_zero}
	v+v^{\perp}=0 \textup{ on } \partial K.
	\end{align}
	
	Next,  for all $\bm w\in\bm W(K)$, by equations \eqref{integration_by_parts1}, \eqref{uniqueness_1} and  \eqref{sum_equa_zero}, we get 
	\begin{align}\label{sum_equa_zero1}
	\begin{split}
	(\nabla\times v,\bm w)_K&=-\langle v,\bm n\times\bm w \rangle_{\partial K} + (v, \nabla\times \bm w)_K\\
	&=-\langle v,\bm n\times\bm w \rangle_{\partial K} \\
	&=\langle v^{\perp},\bm n\times\bm w \rangle_{\partial K}\\
	&=-(\nabla\times v^{\perp},\bm w)_K+(v^{\perp},\nabla\times\bm w)_K.
	\end{split}
	\end{align}
	By equation \eqref{M_decomposition_2}, we know $\nabla\times \bm  w \in \widetilde V(K)$. Moreover, $v^\perp\in \widetilde  V^\perp(K)$, hence we have 
	\begin{align}\label{sum_equa_zero2}
	(v^{\perp},\nabla\times\bm w)_K=0.
	\end{align}
	We combine equations \eqref{sum_equa_zero1} and \eqref{sum_equa_zero2} to get
	\begin{align}\label{important_proof_1}
	(\nabla\times (v+v^{\perp}),\bm w)_K=0\text{ for all }\bm w\in\bm W(K).
	\end{align}
	
	Finally, since $v\in \widetilde V(K)$ and  $v^{\perp}\in \widetilde V^{\perp}(K)$, then by the equation \eqref{M_decomposition_2} we obtain  $\nabla\times (v+v^{\perp})\in \widetilde{\bm W}(K)\subset\bm W(K)$. We take 
	$\bm w = \nabla\times (v+v^{\perp})$ in the equation \eqref{important_proof_1} to get 
	\begin{align}\label{important_proof_2}
	\nabla\times (v+v^{\perp})=\bm 0.
	\end{align}
	Combine equations \eqref{sum_equa_zero} and \eqref{important_proof_2} to get
	\begin{align*}
	v+v^{\perp}=0\textup{ on } K.
	\end{align*}
	Since $v\in \widetilde V(K)$ and  $v^{\perp}\in \widetilde V^{\perp}(K)$, it follows that $v=-v^{\perp}=0$.
	This proves that $[\nabla\times\bm W(K)]^{\perp}\cap \widetilde{V}(K)=\{0\}$, i.e., the space $\widetilde{V}(K)$ is unique and $ \widetilde{V}(K)\subset \left[[\nabla\times\bm W(K)]^{\perp}\right]^\perp = \nabla\times\bm W(K)$. Furthermore, 
	by the equation \eqref{M_decomposition_2}, we have 
	\begin{align*}
	\widetilde{V}(K)=\nabla\times\bm W(K).
	\end{align*}	
\end{proof}

In \Cref{uniquness_of_widetilde_V}, we have proved that if $ V(K)\times\bm W(K)$ admits an $M$-decomposition with associated spaces $\widetilde{ V}(K)$ and $\widetilde{\bm W}(K)$, then $\widetilde V(K)$ is unique and $\widetilde V(K) = \nabla\times \bm W(K)$. However, we do not have similar characterization of the space $\widetilde{\bm W}(K)$, i.e., the space $\widetilde{\bm W}(K)$ is not unique. Hence, it is important to provide a   ``canonical'' $M$-decomposition. 

First, we define the following space:
\begin{align}\label{definition_W0}
\bm W_{0}(K)=\{\bm w\in \bm W(K):\nabla\times\bm w=0\mbox{ in }K,\bm n\times\bm w\times\bm n=\bm 0
\mbox{ on }\partial K \}.
\end{align}

\begin{lemma} \label{can_M_decomposition_theorem} 
	[The canonical $M$-decomposition]
	Let $ V(K)\times\bm W(K)$ admit an $M$-decomposition with associated spaces $\widetilde{ V}(K)$ and $\widetilde{\bm W}(K)$,
	then $ V(K)\times\bm W(K)$ admits an $M$-decomposition
	with associated spaces $\widetilde{ V} (K)$ and $\widetilde{\bm W}_c(K)$, in which 
	\begin{align}\label{can_M_decomposition_theorem_1}
	\widetilde{V}(K)=\nabla\times\bm W(K),\qquad\widetilde{\bm W}_c(K)=\nabla\times V(K)\oplus\bm W_{0}(K).
	\end{align}
	In this case, we say that $ V(K)\times\bm W(K)$ admits the canonical $M$-decomposition
	with associated spaces $\widetilde{ V} (K)$ and $\widetilde{\bm W}_c(K)$.
\end{lemma}
The proof of \Cref{can_M_decomposition_theorem} follows from \Cref{uniquness_of_widetilde_V,inclusion_1,M_decomposition_X_perp,inclusion_2}. 

Next, we define the following space:
\begin{align*}
\bm X(K)=\{\widetilde{\bm w}^{\perp}-\bm{\Pi}_0\widetilde{\bm w}^{\perp}:\widetilde{\bm w}^{\perp}\in \widetilde{\bm W}^{\perp}(K) \},
\end{align*}
where $\bm{\Pi}_0$ is the $L^2$-projection onto the space $\bm W_0(K)$.

\begin{lemma}\label{inclusion_1}
	If $ V(K)\times\bm W(K)$ admits an $M$-decomposition with associated spaces $\widetilde{ V}(K)$ and $\widetilde{\bm W}(K)$, then we have $\widetilde{\bm W}_c(K) = \nabla\times V(K)\oplus\bm W_{0}(K) \subset \bm X^{\perp}(K)$.	
\end{lemma}
\begin{proof}
	For any  $\bm w\in\nabla\times V(K)\oplus\bm W_0(K)$ and $\widetilde{\bm w}^{\perp}\in \widetilde{\bm W}^{\perp}(K)$, then there exist $v\in V(K)$
	and $\bm r\in \bm W_0(K)$ such that  $\bm w=\nabla\times v+\bm r$. Hence, we have 
	\begin{align}\label{proof_inclusion_1}
	\begin{split}
	(\bm w,\widetilde{\bm w}^{\perp}-\bm{\Pi}_0\widetilde{\bm w}^{\perp})_K&=(\nabla\times v,\widetilde{\bm w}^{\perp}-\bm{\Pi}_0\widetilde{\bm w}^{\perp})_K+(\bm r,\widetilde{\bm w}^{\perp}-\bm{\Pi}_0\widetilde{\bm w}^{\perp})_K\\
	&=(\nabla\times v,\widetilde{\bm w}^{\perp})_K-(\nabla\times v,\bm{\Pi}_0\widetilde{\bm w}^{\perp})_K+(\bm r,\widetilde{\bm w}^{\perp}-\bm{\Pi}_0\widetilde{\bm w}^{\perp})_K\\
	&=(\nabla\times v,\widetilde{\bm w}^{\perp})_K-(\nabla\times v,\bm{\Pi}_0\widetilde{\bm w}^{\perp})_K\\
	&=-(\nabla\times v,\bm{\Pi}_0\widetilde{\bm w}^{\perp})_K,\\
	\end{split}
	\end{align}
	where the last equality follows from definition \eqref{M_decomposition_1}.
	
	Next, we estimate $(\nabla\times v,\bm{\Pi}_0\widetilde{\bm w}^{\perp})_K$, since  $ \bm{\Pi}_0\widetilde{\bm w}^{\perp} \in \bm W_0(K)$, by \eqref{integration_by_parts1} to obtain
	\begin{align}\label{proof_inclusion_3}
	(\nabla\times v,\bm{\Pi}_0\widetilde{\bm w}^{\perp})_K=(v,\nabla\times\bm{\Pi}_0\widetilde{\bm w}^{\perp})_K-\langle v,\bm n\times \bm{\Pi}_0\widetilde{\bm w}^{\perp} \rangle_{\partial K}=0.
	\end{align}
	
	Combining  \eqref{proof_inclusion_1} and  \eqref{proof_inclusion_3} show that 
	\begin{align*}
	(\bm w,\widetilde{\bm w}^{\perp}-\bm{\Pi}_0\widetilde{\bm w}^{\perp})_K=0.
	\end{align*}
	This proves $\bm w\in \bm X^{\perp}(K)$, i.e., $\nabla\times V(K)\oplus\bm W_{0}(K) \subset \bm X^{\perp}(K)$.	
\end{proof}

\begin{lemma}\label{M_decomposition_X_perp}
	If $ V(K)\times\bm W(K)$ admits an $M$-decomposition with associated spaces $\widetilde{ V}(K)$ and $\widetilde{\bm W}(K)$, then $ V(K)\times\bm W(K)$ admits an $M$-decomposition with associated spaces $\widetilde{ V} (K)$ and $\bm X^{\perp}(K)$.
\end{lemma}

\begin{proof}
	By the \Cref{def_M_decomposition}, we need to check conditions \eqref{M_decomposition_1}-\eqref{M_decomposition_4}. It is obvious to see that \eqref{M_decomposition_1} and \eqref{M_decomposition_4} hold. Hence, we only need to check conditions \eqref{M_decomposition_2} and \eqref{M_decomposition_3}.
	
	First, by the \Cref{inclusion_1}, we have  $\nabla\times V(K)\oplus\bm W_{0}(K) \subset \bm X^{\perp}(K)$, which implies  $\nabla\times V(K)\subset \bm X^{\perp}(K)$. Moreover, it is obvious that $\nabla \times \bm  W(K)\subset \widetilde V(K)$. This proves  condition \eqref{M_decomposition_2}.	
	
	Next, we  check  condition \eqref{M_decomposition_3}. By the definition of $\bm X(K)$, it is obvious from the definition of $\bm W_0(K)$ in \eqref{definition_W0} that 
	\begin{align}\label{Proof_M_decomposition_X_perp_1}
	\gamma\bm X(K)=\gamma\widetilde{\bm W}^{\perp}(K).
	\end{align}
	Since $\bm \Pi_0$ is the $L^2$-projection onto space $\bm W_0(K)$ and we already proved that $\bm W_{0}(K) \subset \bm X^{\perp}(K)$ in \Cref{inclusion_1}, we can use the definition of $\bm X(K)$ to get 
	\begin{align}\label{Proof_M_decomposition_X_perp_2}
	\dim \bm X(K)\le \dim \widetilde{\bm W}^{\perp}(K).
	\end{align}
	Since $ V(K)\times\bm W(K)$ admits an $M$-decomposition with associated spaces $\widetilde{ V}(K)$ and $\widetilde{\bm W}(K)$, then by \eqref{M_decomposition_3} it holds that: ${\rm tr}:\left( \widetilde{ V} ^{\perp}(K)\times{\widetilde {\bm W}}^{\perp}(K)\right) \to \bm M(\partial K)$ is an isomorphism. Then by \eqref{Proof_M_decomposition_X_perp_1} and \eqref{Proof_M_decomposition_X_perp_2} it follows that  ${\rm tr}:\left( \widetilde{ V} ^{\perp}(K)\times{\bm X}(K)\right) \to \bm M(\partial K)$ is an isomorphism. Since $\bm X(K)$ and $\bm X^{\perp}(K)$ are complete, then $\bm X(K)=[\bm X^{\perp}(K)]^{\perp}$. Therefore, $ V(K)\times\bm W(K)$ admit an $M$-decomposition
	with associated spaces $\widetilde{ V} (K)$ and $\bm X^{\perp}(K)$.

\end{proof}

\begin{lemma}\label{inclusion_2}
	If $ V(K)\times\bm W(K)$ admits an $M$-decomposition with associated spaces $\widetilde{ V}(K)$ and $\widetilde{\bm W}(K)$, then we have $\bm X^{\perp}(K)\subset\widetilde{\bm W}_c(K) = \nabla\times V(K)\oplus\bm W_{0}(K)$.
\end{lemma}

The proof of \Cref{inclusion_2} is similar to that of \Cref{uniquness_of_widetilde_V} and hence we omit it here.

\begin{lemma}\label{dimension_equal_V_W} 
	Let $ V(K)\times\bm W(K)$ admit the canonical $M$-decomposition
	with associated spaces $\widetilde{ V} (K)$ and $\widetilde{\bm W}_c(K)$, then there holds
	\begin{align}\label{dimension_equal_V_W_1}
	\dim\gamma\widetilde{V}^{\perp}(K)=\dim\widetilde{V}^{\perp}(K), \qquad
	\dim\gamma\widetilde{\bm W}_c^{\perp}(K)=\dim\widetilde{\bm W}_c^{\perp}(K).
	\end{align}
\end{lemma}
\begin{proof} 
	Let $v^{\perp}\in \widetilde{V}^\perp(K)$, if $v^{\perp}=0$ on $K$, then obviously we have $\bm n\times v^{\perp}=\bm 0$ on $\partial K$; if $\bm n\times v^{\perp}=\bm 0$ on $\partial K$, then it holds $v^{\perp}=0$ on $\partial K$. 
	For any $\bm w\in \bm W(K)$, by condition \eqref{M_decomposition_2} we know $\nabla \times \bm  w \in \widetilde{V}^\perp(K)$, which implies $(v^{\perp},\nabla \times \bm w)_K=0$. Hence, integration by parts leads to
	\begin{align}\label{dimension_equal_V_W_proof_1}
	0=\langle\bm n\times v^{\perp},\bm w \rangle_{\partial K}=(\nabla\times v^{\perp},\bm w)_K-(v^{\perp},\nabla \times \bm w)_K	=(\nabla\times v^{\perp},\bm w)_K.
	\end{align}
	We take $\bm w=\nabla\times v^{\perp}$ in \eqref{dimension_equal_V_W_proof_1} to get $\nabla\times v^{\perp}=0$ on $K$, combining this with  $v^{\perp}=0$ on $\partial K$ imply  $v^{\perp}=0$ on $K$.
	
	Let $\bm w^{\perp}\in \widetilde{\bm W}_c ^{\perp}(K)$, if $\bm w^{\perp}=\bm 0$ on $K$, then we have $\bm n\times\bm w^{\perp}\times\bm n=\bm 0$ on $\partial K$; if $\bm n\times\bm w^{\perp}\times\bm n=\bm 0$, for any $v\in V(K)$, we have
	\begin{align}\label{dimension_equal_V_W_proof_2} 
	\begin{split}
	0&=\langle\bm n\times\bm w^{\perp}\times\bm n,v\times\bm n \rangle_{\partial K}=\langle\bm w^{\perp}-(\bm w^{\perp}\cdot\bm n)\bm n,v\times\bm n \rangle_{\partial K}\\
	&=\langle\bm n\times\bm w^{\perp},v \rangle_{\partial K}=(\nabla\times\bm w^{\perp},v)_K-(\bm w^{\perp},\nabla\times v)_K\\
	&=(\nabla\times\bm w^{\perp},v)_K.
	\end{split}
	\end{align}
	We take $v=\nabla\times\bm w^{\perp}\in V(K)$ in \eqref{dimension_equal_V_W_proof_2} to get $\nabla\times\bm w^{\perp}=0$ on $K$, then combining this with  $\bm n\times\bm w^{\perp}\times\bm n=\bm 0$ on $\partial K$ to get $\bm w^{\perp}\in\bm W_0(K)$. 
	By \eqref{can_M_decomposition_theorem_1}, we have $\bm w^{\perp}\in  \widetilde{\bm W}_c(K)$, remembering that $\bm w^{\perp}\in \widetilde{\bm W}_c ^{\perp}(K)$, then $\bm w^{\perp}=\bm 0$.
	
\end{proof}

%
%
%
%
%
%

\subsection{Proof of \Cref{index_IM_M_decomposition}}
\label{Characterization_M_decompositions}
\begin{proof}
	\textbf{The proof of only if:}
	Let $ V(K)\times\bm W(K)$ admits an $M$-decomposition, then by the \Cref{def_M_decomposition} we know  \eqref{iM1}, \eqref{iM2} and \eqref{iM4} hold. Moreover, \Cref{can_M_decomposition_theorem} implies that 
	$ V(K)\times\bm W(K)$ admits the canonical  $M$-decomposition with the associated  spaces $\widetilde V(K)$ and $\bm {\widetilde W}_c(K)$.  By \eqref{M_decomposition_3} and \eqref{dimension_equal_V_W_1} we have
	\begin{align}\label{dimension_imp}
	\dim V(K)+\dim\bm W(K)=\dim\widetilde{ V}(K)+\dim\widetilde{\bm W}_c(K)+\dim \bm M(\partial K).
	\end{align}
	This gives
	\begin{align*}
	\dim \bm M(\partial K)=	\dim V(K)+\dim\bm W(K)-\dim\widetilde{ V}(K)-\dim\widetilde{\bm W}_c(K).
	\end{align*}
	Then we use \eqref{uniquness_of_widetilde_V_1} and \eqref{can_M_decomposition_theorem_1} to get
	\begin{align*}
	\dim \bm M(\partial K)&=	\dim V(K)+\dim\bm W(K)-\dim\nabla\times\bm W(K)-\dim[\nabla\times V(K)\oplus\bm W_0(K)]\\
	&=[\dim V(K)-\dim\nabla\times V(K)]+[\dim\bm W(K)-\dim\nabla\times\bm W(K)-\dim\bm W_0(K)]\\
	&=\dim\{\bm n\times v|_{\partial K}:v\in V(K),\nabla\times v=\bm 0\}\\
	&\quad+\dim\{\bm n\times\bm w\times\bm n|_{\partial K}:\bm w\in\bm W(K),\nabla\times\bm w=0\},
	\end{align*}
	this proves \eqref{iM3}.
	
	\textbf{The proof of if:} If \eqref{iM1}, \eqref{iM2}, \eqref{iM3} and \eqref{iM4} hold, then we only need to prove \eqref{M_decomposition_2} and \eqref{M_decomposition_3}. {If} we define $\widetilde V(K)$ and $\widetilde{\bm W}(K)$ as
	\begin{align*}
	\widetilde{V}(K)=\nabla\times\bm W(K),
	\qquad
	\widetilde{\bm W}(K)=\nabla\times V(K)\oplus\bm W_0(K),
	\end{align*}
	then \eqref{M_decomposition_2} holds. 
	
	Next, we prove that \eqref{M_decomposition_3} hold. Use $I_M(V(K)\times\bm W(K))=0$ to get
	\begin{align*}
	0&=I_M(V(K)\times\bm W(K))\\
	&=\dim \bm M(\partial K)-\dim\{\bm n\times v|_{\partial K}:v\in V(K),\nabla\times v=\bm 0\}\\
	&\quad-\dim\{\bm n\times\bm w\times\bm n|_{\partial K}:\bm w\in\bm W(K),\nabla\times\bm w=0\}\\
	&=\dim \bm M(\partial K)-[\dim V(K)-\dim\nabla\times V(K)]-[\dim\bm W(K)-\dim\nabla\times\bm W(K)-\dim\bm W_0(K)]\\
	&=\dim \bm M(\partial K)-[\dim V(K)+\dim\bm W(K)-\dim\nabla\times\bm W(K)-\dim[\nabla\times V(K)\oplus\bm W_0(K)]\\
	&=\dim \bm M(\partial K) +\dim\widetilde{ V}(K)+\dim\widetilde{\bm W}(K)-\dim V(K)-\dim\bm W(K),
	\end{align*}
	which combines with \eqref{dimension_equal_V_W_1} to get \eqref{M_decomposition_3}. This finishes our proof.
	
\end{proof}

\section{Error Analysis}
\label{Error_estimates}

In this section, we present our main result, an error analysis for the HDG approximation to Maxwell's equations given by \eqref{Maxwell_equation_HDG_form_ori}.  To simplify the derivation we shall assume that $\mu_r$ and
$\epsilon_r$ are constants. {First we discuss the extra conditions on the spaces $V(K)$ and $W(K)$ needed for 
	this analysis.  These conditions arise because at this point each element $K\in{\cal T}_h$ is a general polygon, yet we need certain  properties for functions in these spaces (that hold for standard elements including triangles, parallelograms and squares that are considered later in this paper). For triangles these conditions follow if
	the mesh is assumed to be regular, and the spaces $V(K)$ and $\bm W(K)$ are sufficiently rich. After this discussion, we consider an adjoint problem needed for the analysis and finally present the error analysis.}

\subsection{{Additional assumptions on the approximation spaces}}\label{messy_discussion}
{Throughout this section we
	assume that the following conditions on the local spaces $V(K)$ and $\bm W(K)$ hold:}
\begin{enumerate}
	\item {Most importantly, we assume that the space $V(K)\times \bm W(K)$ admits an $M$-decomposition.}
	\item {The spaces $V(K)$ and $\bm W(K)$ must satisfy
		\begin{equation}
		\mathcal {P}_0(K)\in \widetilde{ V}(K)\mbox{ and }[\mathcal {P}_0(K)]^2\in \widetilde{\bm W}(K),\label{approx-prop}
		\end{equation}
		for all elements $K$.  In addition, we assume that if $\Pi_0$ (respectively $\bm\Pi_0$) denotes the $L^2(K)$ (respectively $\bm L^2(K)$) orthogonal projection onto $V(K)$ (respectively $\bm W(K)$) then the following estimates hold:
		\[
		\Vert \bm w-\bm \Pi_0\bm w\Vert_K\leq Ch_K^s\Vert \bm w\Vert_{\bm H^s(K)}
		\mbox{ and }
		\Vert p-\Pi_0p\Vert_K\leq Ch_K^s\Vert p\Vert_{H^s(K)}
		\]
		for any sufficiently smooth $\bm w$ or $p$ and $0\leq s\leq 1$.}
	\item {Let $\mathcal T_h^\star$ be a refined mesh of $\mathcal T_h$ consisting of simplices obtained by subdividing each element $K\in \mathcal T_h$ using triangles.  We assume that the number of triangles used in each element is bounded independent of $h$ (i.e. there is a fixed maximum number of triangles covering each $K$ independent of $h$). Next we define $\bm W_h^\star = \{\bm u \in \bm L^2(\Omega): \bm u|_K\in   [{\mathcal P}_\ell(K)]^2, \forall K\in\mathcal T_h^\star\}$ and $\ell\ge1$ is some integer  such that $\bm W_h\subset\bm W_h^\star$.  We assume that $\mathcal T_h^\star$ is shape-regular.  This assumption implies that standard scaling estimates
		can be used for $V(K)$ and $\bm W(K)$, because scaling can be used triangle by triangle on the $\mathcal T_h^\star$. In addition, standard finite element spaces constructed on this mesh have the usual approximation properties.}
	
	{Note that this notion of shape regularity for the general mesh is 
		what is the analogue of that used to define shape regularity for a  quadrilateral mesh in~\cite{MR851383}. }
	
\end{enumerate}    

\subsection{The dual problem}
Consider the following dual problem: find $({\Psi},\bm\Phi)\in H({\rm curl};\Omega)\times [\bm H_0({\rm curl};\Omega)\cap\bm H({\rm div}^0;\Omega)]$ such that 
\begin{subequations}\label{dual_Mixed}
	\begin{align}
	\mu_r \Psi-\nabla\times\bm \Phi&=\bm 0&\text{ in }\Omega,\label{dual_Mixed_1}\\
	\nabla\times \Psi -\kappa^2\overline{\epsilon_r}\bm \Phi&=\bm{\Theta}&\text{ in }\Omega,\label{dual_Mixed_2}\\
	\bm{n}\times \bm{\Phi} &=\bm{0}&\text{ on }\Gamma,\label{dual_Mixed_3}
	\end{align}
\end{subequations}
where $\bm{\Theta}\in \bm H({\rm div};\Omega)$ and $\nabla\cdot\bm{\Theta}=0$ and $\overline{\epsilon_r}$ is the
complex conjugate of $\epsilon_r$.
Under our assumptions on {$\mu_r$, $\epsilon_r$ and $\kappa$}, this problem has a unique solution. The regularity of the solution of \eqref{dual_Mixed} is given in \Cref{regularity_dual_pro}.


We recall the following result, where $L_0^2(\Omega)$ denotes the space of functions in $L^2(\Omega)$ with average value zero.	
\begin{lemma}[c.f {\cite[Corollary 2.4]{MR851383}}] \label{lbb1}
	Let $\Omega$ be a bounded connected Lipschitz domain in $\mathbb R^2$, then for any $f\in L^2_0(\Omega)$, there exists a $\bm v\in \bm H_0^1(\Omega)$ such that
	\begin{align*}
	\nabla\cdot\bm v=f,\qquad \|\bm v\|_{\bm H^1(\Omega)}\le C\|f\|_{L^2(\Omega)}.
	\end{align*}
\end{lemma}
With the above result, we are ready to prove the following lemma:
\begin{lemma}\label{lbb2} Let $\Omega$ be a bounded connected Lipschitz domain in $\mathbb R^2$, then for any $f\in L^2(\Omega)$, there exists a $\bm v\in \bm H^1(\Omega)$ such that
	\begin{align*}
	\nabla\cdot\bm v=f,\qquad \|\bm v\|_{\bm H^1(\Omega)}\le C\|f\|_{L^2(\Omega)}.
	\end{align*}
\end{lemma}
\begin{proof} 
	Let $\bar f = |\Omega|^{-1}(f,1)_{\Omega}$ be the mean value of $f$, then $f-\bar f\in L_0^2(\Omega)$. By \Cref{lbb1},  there exists a $\bm w=(w_1, w_2)^T\in \bm H_{0}^1(\Omega)$, such that $\nabla\cdot\bm w=f-\bar f$, and $\|\bm w\|_{\bm H^1(\Omega)}\le C\|f-\bar f\|_{L^2(\Omega)}$. 
	
	Let $(x_0, y_0)$ be a point in the domain $\Omega$. Define $\bm v=(v_1,v_2)^T$ with $v_1=w_1+(x-x_0)\bar f$, $v_2=w_2$, then $\bm v\in \bm H^1(\Omega)$ and 
	\begin{align*}
	\nabla\cdot\bm v&=\nabla\cdot\bm w+\nabla\cdot((x-x_0)\bar f,0)^T=f,\\
	\|\bm v\|_{\bm H^1(\Omega)}&\le 
	\|\bm w\|_{\bm H^1(\Omega)}+\|(x-x_0)\bar f\|_{\bm H^1(\Omega)}\\
	&\le C\|f-\bar f\|_{L^2(\Omega)}+\|(x-x_0)\bar f\|_{L^2(\Omega)}+\|\bar f\|_{L^2(\Omega)}\\
	&\le  C(\|f\|_{L^2(\Omega)}+\|\bar f\|_{L^2(\Omega)})\\
	&\le C\|f\|_{L^2(\Omega)}.
	\end{align*}
\end{proof}
The previous result can be used to prove the existence of a vector potential as follows

\begin{lemma}\label{le41} Let $f\in  L^2(\Omega)$, then there exists a function $\bm w\in \bm H^1(\Omega)$, such that
	\begin{align*}
	\nabla\times \bm w= f, \qquad \|\bm w\|_{\bm H^1(\Omega)}\le  C\|f\|_{L^2(\Omega)}.
	\end{align*}
\end{lemma}
\begin{proof} By \Cref{lbb2}, there exists a $\bm v=(v_1,v_2)^T\in \bm H^1(\Omega)$, such that $\nabla\cdot\bm v=f$ and $\|\bm v\|_{\bm H^1(\Omega)}\le C\|f\|_{L^2(\Omega)}$.
	We take $\bm w=(w_1,w_2)^T$ with $w_1=-v_2$ and $w_2=w_1$, hence
	\begin{align*}
	\nabla\times\bm w&=-\partial_yw_1+\partial_xw_2=\partial_yv_2+\partial_xv_1=\nabla\cdot\bm v=f,\\
	\|\bm w\|_{\bm H^1(\Omega)}&=\|\bm v\|_{\bm H^1(\Omega)}\le C\|f\|_{L^2(\Omega)}.
	\end{align*}
\end{proof}

The proof of the next theorem follows that of \cite[Proposition 3.7]{MR1626990} and \cite[Theorem 3.50]{Monk_Maxwell_Book_2003}, which deal with the 3D case.	
\begin{theorem} \label{th2} Let $\Omega$ be a simply connected Lipschitz domain in $\mathbb R^2$, then the space $ \bm H_0({\rm curl};\Omega)\cap\bm H({\rm div};\Omega)$ is imbedded in the space $\bm H^s(\Omega)$ with some $s\in (\frac{1}{2},1]$, and the following estimate holds
	\begin{align}
	\|\bm u\|_{\bm H^s(\Omega)}\le C\left(
	\|\nabla\times\bm u\|_{L^2(\Omega)}
	+\|\nabla\cdot\bm u\|_{L^2(\Omega)}
	\right),
	\end{align}
	for all $\bm u\in\bm H_0({\rm curl};\Omega)\cap\bm H({\rm div};\Omega)$.		
\end{theorem}
\begin{proof} 
	Let  $\mathcal O$ be  a smooth open set  with a connected boundary (a circle for instance), which contains $\bar{\Omega}$.
	Let $\bm u\in \bm H_0({\rm curl};\Omega)\cap\bm H({\rm div};\Omega)$, we  extend $\bm u$ to $\mathcal O$ by zero, so $\bm u\in \bm H({\rm curl};\mathcal O)$, therefore $\nabla\times\bm u\in L^2(\mathcal O)$.
	From \Cref{le41}, there exists a function $\bm w\in \bm H^1(\mathcal O)$ such that $\nabla\times\bm w=\nabla\times\bm u$ in $\mathcal O$ and $\|\bm w\|_{\bm H^1(\mathcal O)}\le C\|\nabla\times\bm u\|_{L^2(\mathcal O)}$. Since $\nabla\times(\bm u-\bm w)=0$ in $\mathcal O$, then there is a function $\chi\in H^1(\mathcal O)$ such that $\bm u-\bm w=\nabla\chi$ in $\mathcal O$. Since $\bm u=\bm 0$ in $\mathcal O/\bar{\Omega}$, we have $-\bm w=\nabla\chi$ in $\mathcal O\setminus\bar{\Omega}$, therefore, $\chi\in H^2(\mathcal O\setminus \bar{\Omega})$. Then  $\chi\in H^1(\Omega)\setminus\mathbb R$ satisfies 
	\begin{align*}
	\Delta\chi&=\nabla\cdot\bm u-\nabla\cdot\bm w\quad\text{ in }\Omega,
	\end{align*}
	where $\nabla\chi|_{\partial\Omega}$ takes the exterior value of $\nabla\chi=-\bm w$. So $\chi|_{\partial\Omega}\in H^{\frac{1}{2}+s}(\partial\Omega)$ with some $s\in (\frac{1}{2},1]$.
	By \cite[Corollary 18.15]{MR961439}, we have $\chi\in H^{1+s}(\Omega)$  and
	\begin{align*}
	\|\chi\|_{H^{1+s}(\Omega)}&\le C
	\left(\|\nabla\cdot\bm u-\nabla\cdot \bm w\|_{H^{-1+s}(\Omega)}+\|\nabla\chi\|_{H^{-\frac{1}{2}+s}(\partial \Omega)}\right).
	\end{align*}
	Therefore, using the fact that $\nabla\chi=-\bm w$ on $\mathcal O\setminus\bar{\Omega}$ it holds
	\begin{align*}
	\|\bm u\|_{\bm H^{s}(\Omega)}&=\|\bm w+\nabla\chi\|_{\bm H^{s}(\Omega)}\nonumber\\
	&\le \|\bm w\|_{\bm H^1(\Omega)}
	+C
	\left(
	\|\nabla\cdot\bm u-\nabla\cdot\bm w\|_{H^{-1+s}(\Omega)}+\|\nabla\chi\|_{H^{-\frac{1}{2}+s}(\partial\Omega)}
	\right)
	\nonumber\\
	&\le C\left(
	\|\bm w\|_{\bm H^1(\Omega)}+\|\nabla\cdot\bm u\|_{L^2(\Omega)}
	\right)\nonumber\\
	&\le C\left(
	\|\bm w\|_{\bm H^1(\mathcal O)}+\|\nabla\cdot\bm u\|_{L^2(\Omega)}
	\right)\nonumber\\
	&\le C\left(
	\|\nabla\times\bm u\|_{L^2(\Omega)}
	+\|\nabla\cdot\bm u\|_{L^2(\Omega)}
	\right).
	\end{align*}
	Thus we finish our proof.
\end{proof}

Now we can state a complete regularity result for the adjoint problem:

\begin{theorem}\label{regularity_dual_pro} 
	Let $\mu_r$ be a smooth function, then we have the following regularity for the solution of problem \eqref{dual_Mixed}
	\begin{align}\label{regularity_dual}
	\|\Psi\|_{H^1(\Omega)}+\|\bm{\Phi}\|_{\bm H^s(\Omega)}\le C\|\bm{\Theta}\|_{\bm L^2(\Omega)},
	\end{align}
	for some  $s\in (\frac{1}{2},1]$ depending on $\Omega$.
\end{theorem}
\begin{proof} 
	To simplify the notation, we define
	\begin{align*}
	a^+(\bm u,\bm v)=(\mu_r^{-1}\nabla\times\bm u,\nabla\times\bm v)_{\Omega}+(\bm u,\bm v)_{\Omega}.
	\end{align*}
	Let $\widetilde{\bm{\Phi}}\in \bm H_0({\rm curl};\Omega)\cap\bm H({\rm div^0};\Omega)$ be the solution of
	\begin{align}\label{def1}
	a^+(\widetilde{\bm{\Phi}},\bm v)=(\bm \Theta,\bm v)_{\Omega},\qquad \forall \bm v\in \bm H_0({\rm curl};\Omega)\cap\bm H({\rm div^0};\Omega).
	\end{align}
	Let $\mathcal K: \bm L^2(\Omega)\to\bm H_0({\rm curl};\Omega)\cap\bm H({\rm div^0};\Omega)$ be such that for any $\bm w\in \bm L^2(\Omega)$ the function  $\mathcal K\bm w$ satisfies
	\begin{align}\label{def2}
	a^+(\mathcal K \bm w,\bm v)=-(\kappa^2\overline{\epsilon_r}-1)(\bm w,\bm v)_{\Omega},\qquad \forall \bm v\in \bm H_0({\rm curl};\Omega)\cap\bm H({\rm div^0};\Omega).
	\end{align}
	Obviously,  $\widetilde{\bm{\Phi}}$ and $\mathcal K $ are well-defined and
	\begin{align*}
	a^+((\mathcal I+\mathcal K )\bm{\Phi},\bm v)=a^+(\widetilde{\bm{\Phi}},\bm v).
	\end{align*}
	This gives
	\begin{align}\label{eq}
	(\mathcal I+\mathcal K )\bm{\Phi}=\widetilde{\bm{\Phi}}.
	\end{align}
	Form \eqref{def1} and \eqref{def2}, we get
	\begin{align*}
	\|\widetilde{\bm{\Phi}}\|_{\bm L^2(\Omega)}+\|\nabla\times\widetilde{\bm{\Phi}}\|_{L^2(\Omega)}&\le C\|\bm{\Theta}\|_{\bm L^2(\Omega)},\\
	\|\mathcal K \bm{\Phi}\|_{\bm L^2(\Omega)}+\|\nabla\times\mathcal K \bm{\Phi}\|_{L^2(\Omega)}&\le C\|\bm{\Phi}\|_{\bm L^2(\Omega)}.
	\end{align*}
	From \Cref{th2}, we know that $ \bm H_0({\rm curl};\Omega)\cap\bm H({\rm div^0};\Omega)$ is compactly imbedded in the space  $\bm L^2(\Omega)$. Using \eqref{def2} we see that $\mathcal K $ is a self-adjoint and compact operator on $\bm L^2(\Omega)$. Hence,  since our assumptions on $\epsilon_r$ and $\kappa^2$ guarantee at most one solution, by the Fredholm Alternative,  \eqref{eq} has a unique solution. Therefore,
	\begin{align*}
	\|{\bm{\Phi}}\|_{\bm L^2(\Omega)}&=
	\|(\mathcal I+\mathcal K )^{-1}\widetilde{\bm{\Phi}}\|_{\bm L^2(\Omega)}\le C\|\bm{\Theta}\|_{\bm L^2(\Omega)},\\
	\|\nabla\times\bm{\Phi}\|_{L^2(\Omega)}&\le \|\nabla\times\mathcal K\bm{\Phi}\|_{L^2(\Omega)}
	+ \|\nabla\times\widetilde{\bm{\Phi}}\|_{L^2(\Omega)}\\
	&\le C\left(\|\bm{\Theta}\|_{\bm L^2(\Omega)}+\|\bm{\Phi}\|_{\bm L^2(\Omega)}\right)\le C\|\bm{\Theta}\|_{\bm L^2(\Omega)}.
	\end{align*}
	By the \Cref{dual_Mixed},  we have
	\begin{align*}
	\|\nabla\times(\mu_r^{-1}\nabla\times\bm{\Phi})\|_{\bm L^2(\Omega)}\le C\|\bm{\Theta}\|_{\bm L^2(\Omega)}+C\|\bm{\Phi}\|_{\bm L^2(\Omega)}\le C\|\bm{\Theta}\|_{\bm L^2(\Omega)}.
	\end{align*}
	Since $\mu$ is smooth, then we have $\nabla\times\bm{\Phi}\in H^1(\Omega)$ and
	\begin{align*}
	\|\nabla\times\bm{\Phi}\|_{H^1(\Omega)}\le C\|\bm{\Theta}\|_{\bm L^2(\Omega)}.
	\end{align*}
	Since \Cref{th2} ensures $ \bm H_0({\rm curl};\Omega)\cap\bm H({\rm div^0};\Omega)$ is imbeded in the space $\bm H^s(\Omega)$ with $(\frac{1}{2},1]$, then we have 	
	\begin{align}
	\|\bm{\Phi}\|_{\bm H^s(\Omega)}\le C\left(
	\|\bm{\Phi}\|_{\bm L^2(\Omega)}+\|\nabla\times\bm{\Phi}\|_{L^2(\Omega)}	\right)\le C\|\bm{\Theta}\|_{\bm L^2(\Omega)}. 
	\end{align}
	This finishes our proof.
\end{proof}


Throughout this section, we use $C$ to denote a positive constant independent of mesh size, which may take on different values at each occurrence. Let $P_V$, $P_{\widetilde V}$, $\bm P_{\bm W}$ and $\bm P_{\widetilde{\bm W}}$ are the $L^2$- projections on the spaces $V_h, \widetilde V_h, \bm W_h$ and $\widetilde{\bm W}_h$, respectively.

Now we state the main result in this section. The proof is found in \Cref{proof_of_main_result_2}.
\begin{theorem}\label{main_result_error_analysis} Suppose the spaces $(V_h,\bm W_h,\bm M_h)$ have an $M$-decomposition and the assumptions in Section~\ref{messy_discussion} are satisfied.
	Let $( q,\bm u)\in  H({\rm curl};\Omega)\times \bm H({\rm curl};\Omega)$ and $( q_h,\bm u_h,\widehat{\bm u}_h)\in V_h\times \bm W_h\times \bm M_h$ be the solution of \eqref{Maxwell_equation_mixed_form} and \eqref{Maxwell_equation_HDG_form_ori}, respectively.  Then there exists an $h_0>0$ such that for all $h\le h_0$, we have the error estimate
	\begin{align*}
	\|q- q_h\|_{\mathcal T_h}&\le C\left(\|q - P_V q\|_{\mathcal T_h} + \|\bm u-\bm {P_{W}} \bm u\|_{\mathcal T_h}\right),\\
	\|\bm u - \bm u_h\|_{\mathcal T_h} &\le C\left(\|q - P_V q\|_{\mathcal T_h} + \|\bm u-\bm {P_{W}} \bm u\|_{\mathcal T_h}\right).
	\end{align*}
	Furthermore, the post processed solution $\bm u_h^\star\in \bm W^{\star}(\mathcal T_h)$ defined later in \eqref{post_processing} satisfies  the estimate
	\begin{align*}
	\|\nabla\times( \bm u - \bm u_h^{\star})\|_{\mathcal{T}_h} \le C\left( \|q - P_V q\|_{\mathcal T_h} +  \|\bm u-\bm {P_{W}} \bm u\|_{\mathcal T_h}+\inf_{\bm w_h^\star\in \bm W^{\star}(\mathcal T_h)}\|\nabla\times(\bm u-\bm w_h^\star)\|_{\mathcal{T}_h}\right).
	\end{align*}	
\end{theorem}

\subsection{The HDG Projection}
\label{The_HDG_Projection}
An appropriate HDG projection  plays a key role in the derivation of optimal error estimates and  superconvergence (see for example \cite{Cockburn_Dong_Guzman,Cockburn_Gopalakrishnan_Sayas_Porjection_MathComp_2010,Cockburn_Gopalakrishnan_Nguyen_Peraire_Sayas_Stokes_MathComp_2011,Chabaud_Cockburn_Heat_MathComp_2012,Celiker_Cockburn_Shi_MathComp_2012,Cockburn_Qiu_Shi_MathComp_2012,Chen_Cockburn_Convection_Diffusion_IMAJNA_2012,Cesmelioglu_Cockburn_Nguyen_Peraire_Oseen_JSC_2013}). 
In the case of Maxwell's equations, we define the following  HDG projection: find $({\Pi_V} q, \bm{\Pi_W}\bm u)\in  V(K)\times\bm W(K)$ such that
\begin{subequations}\label{projection}
	\begin{align}
	({\Pi_V} q, v_{h})_K&=( q, v_{h})_K&\forall v_h\in\widetilde{ V}(K), \label{P1}\\
	(\bm{\Pi_W}\bm u, \bm w_{h})_K&=(\bm u, \bm w_{h})_K&\forall\bm w_h\in\widetilde{\bm W}(K),\label{P2}\\
	\langle{\Pi_V} q-\tau\bm n\times\bm{\Pi_W}\bm u, \bm n\times\bm{\mu}_h \rangle_{F}&=\langle q-\tau\bm n\times\bm u, \bm n\times\bm{\mu}_h \rangle_{F}\nonumber\\&\qquad\forall\bm \mu_h\in\bm M(F) \mbox{ and for all edges } F\subset\partial K.\label{P3}
	\end{align}
\end{subequations}
The following theorem proves that the above definition uniquely specifies the projections and provides optimal error estimates for this projection.


\begin{theorem}\label{projection_error}
	System \eqref{projection} defines a unique projection $({\Pi_Vq}, \bm{\Pi_Wu})$. Moreover, we have the following error estimate:
	\begin{subequations}\label{projection_error_11}
		\begin{align}
		\|\bm{\Pi_W\bm u}-\bm u\|_K &\le C\left(\|\bm u-\bm{P_Wu}\|_{ K}+h_K\|\nabla\times q-\bm{P}_{\widetilde{\bm W}}\nabla\times q\|_{K}+h_K^{1/2}\|\bm n \times (\bm u-\bm{P_W}\bm u)\|_{\partial K}\right),\label{projection_error_1}\\
		\|{\Pi_Vq}- q\|_K&\le C \left(h_K^{1/2}\|q-{P_V} q\|_{\partial K}+ h_K^{1/2}\|\bm n\times (\bm{\Pi_Wu}-\bm u)\|_{\partial K}+\| q-{P_Vq}\|_K\right).\label{projection_error_2}
		\end{align}
	\end{subequations}	
\end{theorem}


We only give a proof for \eqref{projection_error_1} in the following three lemmas, since \eqref{projection_error_2} is very similar.

\begin{lemma}[Existence and Uniqueness]\label{uniquness}
	System \eqref{projection} defines a unique projection $({\Pi_Vq}, \bm{\Pi_Wu})$.
\end{lemma}

\begin{proof}
	By \Cref{def_M_decomposition} we have
	\begin{align*}
	\dim V(K)+\dim\bm W(K)=\dim\widetilde{ V}(K)+\dim\widetilde{\bm W}(K)+\dim \bm M(\partial K).
	\end{align*}
	This means that system \eqref{projection} is square, hence we only need to prove uniqueness. We set the right hand sides of \eqref{projection} to zero, i.e., 
	$ q=0$ and $\bm u=\bm 0$. By \eqref{P1} and \eqref{P2}, we have
	\begin{align}\label{well_posedness_1}
	{\Pi_V q}\in\widetilde{ V}^{\perp}(K) \quad \textup{and} \quad \bm{\Pi_Wu}\in\widetilde{\bm W}^{\perp}(K).
	\end{align}
	Since $\bm n\times \bm  W(K)\times \bm n \subset \bm  M(\partial K)$, then we can take $\bm\mu_h=\bm n \times \bm{\Pi_Wu}\times \bm n$ in \eqref{P3} to get
	\begin{align*}
	\langle\tau\bm n\times\bm{\Pi_W}\bm u,\bm n\times(\bm n\times\bm{\Pi_W}\bm u\times\bm n)\rangle_{\partial K}& = \langle\Pi_V q,   \bm n\times(\bm n\times\bm{\Pi_W}\bm u\times\bm n) \rangle_{\partial K}\\
	&=\langle{\Pi_V} q,\bm n\times\bm{\Pi_W}\bm u \rangle_{\partial K}\\
	&=({\Pi_Vq},\nabla\times\bm{\Pi_W}\bm u)_K-(\nabla\times{\Pi_Vq},\bm{\Pi_W}\bm u)_K\\
	&=0.
	\end{align*}
	Since $\tau$ is  piecewise constant and positive, then
	\begin{align}\label{well_posedness_2}
	\bm n\times\bm{\Pi_W}\bm u\times\bm n=\bm 0 \quad \textup{on}\quad  \partial K.
	\end{align}
	Moreover, $\bm n\times\bm{\Pi_W}\bm u=0$ on $\partial K$. Since $\bm n\times V(K) \subset \bm  M(\partial K)$, then we can take  $\bm\mu_h=\bm n\times{\Pi_Vq}$ in \eqref{P3} to get
	\begin{align}\label{well_posedness_3}
	\bm n\times{\Pi_V} q=0\quad \textup{on}\quad  \partial K.
	\end{align}
	
	We combine \eqref{well_posedness_1}, \eqref{well_posedness_2}, \eqref{well_posedness_3} and \eqref{M_decomposition_3} to conclude that $\bm{\Pi_Wu}=\bm 0$ and   ${\Pi_Vq}=0$. This proves the system \eqref{projection} defines a unique projection $({\Pi_Vq},\bm{\Pi_Wu})$.
\end{proof}

To estimate $\bm{\Pi_Wu}-\bm u$, we decouple the projection $\bm{\Pi_W}$ from $\Pi_V$ in \eqref{projection} as follows.
\begin{lemma} The projection
	$\bm{\Pi_Wu}$  satisfies
	\begin{subequations}\label{decoup}
		\begin{align}
		(\bm{\Pi_W}\bm u,\bm v_h)_K&=(\bm u,\bm v_h)_K, &\forall\bm v_h\in \widetilde{\bm W}(K),\label{A1}\\
		\langle\tau\bm n\times \bm{\Pi_Wu}\times\bm n,\bm w_h \rangle_{\partial K}&=(\nabla\times q,\bm w_h)_K+\langle\tau\bm n\times \bm u\times\bm n,\bm w_h \rangle_{\partial K}, & \forall\bm w_h\in\widetilde{\bm W}^{\perp}(K).\label{A2}
		\end{align}
	\end{subequations}
\end{lemma}
\begin{proof}
	Noticing that \eqref{P3} can be rewritten as	
	\begin{align}\label{rewrite_1}
	\langle \tau\bm n\times\bm{\Pi_W}\bm u,\bm n\times\bm \mu_h \rangle_{\partial K}=\langle \Pi_V q - q,\bm n\times\bm \mu_h\rangle_{\partial K} + \langle\tau \bm n \times \bm u ,\bm n\times\bm \mu_h \rangle_{\partial K}.
	\end{align}
	Since $\bm n\times\bm W(K)\times\bm n\subset \bm M(\partial K)$, then we take $\bm \mu_h = \bm n\times \bm w_h\times \bm n$ in \eqref{rewrite_1} to get 
	\begin{align}\label{A3}
	\langle \tau\bm n\times\bm{\Pi_W}\bm u\times \bm n, \bm w_h \rangle_{\partial K}=\langle  \bm n\times (q-\Pi_V q ), \bm w_h\rangle_{\partial K} + \langle\tau \bm n \times \bm u \times \bm n, \bm w_h \rangle_{\partial K}.
	\end{align}
	Then, for all $\bm w_h\in \widetilde{\bm W}^{\perp}(K)$, by \eqref{M_decomposition_2} and \eqref{P1}, we have 
	\begin{subequations}
		\begin{align}
		(\nabla\times{\Pi_Vq},\bm w_h)_K=0,\label{zero}\\
		( q-{\Pi_Vq},\nabla\times\bm w_h)_K=0.\label{pk-1}
		\end{align}	
	\end{subequations}
	Next, we use the integration by parts identity  \eqref{integration_by_parts2} to get
	\begin{align}\label{A4}
	\langle \bm n\times(  q-{\Pi_Vq}  ),\bm w_h\rangle_{\partial K}&=(\nabla\times( q-{\Pi_Vq}),\bm w_h)_K-( q-{\Pi_Vq},\nabla\times\bm w_h)_K\nonumber\\
	&=(\nabla\times( q-{\Pi_Vq}),\bm w_h)_K, &\text{by }\eqref{pk-1}\nonumber\\
	&=(\nabla\times q,\bm w_h)_K.&\text{by }\eqref{zero}
	\end{align}
	Therefore, \eqref{P2}, \eqref{A3} and \eqref{A4} gives the  system \eqref{decoup}.
\end{proof}
Now we can give the proof of \eqref{projection_error_1}.
\begin{proof}[Proof of \eqref{projection_error_1}]
	By the definition of $\bm {P_W}$ and $\bm {P_{\widetilde W}}$,  we can rewrite equation \eqref{decoup} as follows: 
	\begin{subequations}\label{A5}
		\begin{align}
		(\bm{\Pi_W}\bm u-\bm {P_W}\bm u,\bm v_h)_K&=0,&\forall\bm v_h\in\widetilde{\bm W}(K),\label{A5a}\\
		\langle\tau\bm n\times (\bm{\Pi_Wu}-\bm{P_W}\bm u),\bm n\times\bm w_h \rangle_{\partial K}&=(\nabla\times q-\bm P_{\widetilde{\bm W}}\nabla\times q,\bm w_h)_K\nonumber\\
		& \ \ +\langle\tau\bm n\times (\bm u-\bm{P_W}\bm u),\bm n\times\bm w_h \rangle_{\partial K}, &\forall\bm w_h\in\widetilde{\bm W}^{\perp}(K). \label{A5b}
		\end{align}
	\end{subequations}
	By the same arguments as in the proof of \Cref{uniquness}, we can prove that $\bm{\Pi_Wu}-\bm{P_W}\bm u\in\bm W(K)$ is uniquely determined by the right hand side  of \eqref{A5}. {Using a standard scaling estimate (this can be used because of the  assumption on $\mathcal T_h^*$ in Section~\ref{messy_discussion}}) we have
	\begin{align*}
	\|\bm{\Pi_Wu}-\bm{P_W}\bm u\|_K\le C h_K\|\tau^{-1}(\nabla\times q-\bm{P}_{\widetilde{\bm W}}\nabla\times q)\|_{K}+h_K^{1/2}\|\bm n\times (\bm u-\bm{P_W}\bm u)\|_{\partial K}.
	\end{align*}
	Thus, the triangle inequality gives the desired result.	
\end{proof}
%
%

Next, we extend the error estimates \eqref{projection_error_11} to fractional order Sobolev spaces. {To do this we use a local inverse inequality. For any function $\bm w_h\in \bm W(K)$ or $p_h\in V(K)$ the following inverse estimate holds:
	\[
	\Vert \bm w_h\Vert_{H^s(K)}\leq Ch^{-s}_K\Vert\bm w_h\Vert_K,\mbox{ and }
	\Vert \bm p_h\Vert_{H^s(K)}\leq Ch^{-s}_K\Vert\bm p_h\Vert_K
	\]
	with $0\leq s\leq 1$. The constant $C$ is independent of the function, element and mesh size.  Note that this assumption follows from our assumption
	on the auxiliary mesh $\mathcal T_h^*$ when $s=1$ and trivially holds when $s=0$. Hence by interpolation it holds
	for general $0\leq s\leq 1$.}

\begin{lemma}\label{projection_errorf}  For
	any $s\in [0,1]$, we have 
	\begin{subequations}\label{projection_error_11f}
		\begin{align}
		\|\bm{\Pi_W\bm u}-\bm u\|_{\bm H^s(K)} &\le  C h_K^{-s}\left(\|\bm u-\bm{P_Wu}\|_{ K}+h_K\|\nabla\times q-\bm{P}_{\widetilde{\bm W}}\nabla\times q\|_{K}+h_K^{1/2}\|\bm n \times (\bm u-\bm{P_W}\bm u)\|_{\partial K}\right) \nonumber\\
		&\quad + \|\bm P_{\bm W} \bm u-\bm u\|_{\bm H^s(K)},\label{projection_error_1f}\\
		\|{\Pi_Vq}- q\|_{ H^s(K)} &\le C \left(h_K^{1/2-s}\|q-{P_V} q\|_{\partial K}+h_K^{-s}\| q-{P_Vq}\|_K +\| q-{P_Vq}\|_{H^s(K)}+\| \bm u-\bm {P_W}\bm u\|_{\bm H^s(K)}\right)\nonumber\\
		&\quad +C \left(h_K^{-s}\|\bm u-\bm{P_Wu}\|_{ K}+h_K^{1-s}\|\nabla\times q-\bm{P}_{\widetilde{\bm W}}\nabla\times q\|_{K}+h_K^{1/2-s}\|\bm n \times (\bm u-\bm{P_W}\bm u)\|_{\partial K}\right).\label{projection_error_2f}
		\end{align}
	\end{subequations}	
\end{lemma}

\begin{proof}
	Using the fact that $\bm P_{\bm W}$ is the $\bm L^2$ orthogonal projection on $\bm W(K)$ and {applying the   local inverse inequality discussed before the statement of the lemma,} we get 
	\begin{align*}
	\|\bm{\Pi_W\bm u}-\bm u\|_{\bm H^s(K)}  &= \|\bm{\Pi_W\bm u}- \bm P_{\bm W} \bm u + \bm P_{\bm W} \bm u-\bm u\|_{\bm H^s(K)} \\
	&\le \|\bm{\Pi_W\bm u}- \bm P_{\bm W} \bm u\|_{\bm H^s(K)} + \|\bm P_{\bm W} \bm u-\bm u\|_{\bm H^s(K)} \\
	&\le C h_K^{-s}\|\bm{\Pi_W\bm u}- \bm P_{\bm W} \bm u\|_{K} + \|\bm P_{\bm W} \bm u-\bm u\|_{\bm H^s(K)}\\
	&\le  C h_K^{-s}\|\bm{\Pi_W\bm u}-  \bm u\|_{K} +C h_K^{-s}\|\bm P_{\bm W} \bm u - \bm u\|_{K} + \|\bm P_{\bm W} \bm u-\bm u\|_{\bm H^s(K)}.
	\end{align*} 
	Combining the estimate \eqref{projection_error_1} and the above inequality we have
	\begin{align*}
	\|\bm{\Pi_W\bm u}-\bm u\|_{\bm H^s(K)}  &\le  C h_K^{-s}\left(\|\bm u-\bm{P_Wu}\|_{ K}+h_K\|\nabla\times q-\bm{P}_{\widetilde{\bm W}}\nabla\times q\|_{K}+h_K^{1/2}\|\bm n \times (\bm u-\bm{P_W}\bm u)\|_{\partial K}\right) \\
	&\quad + \|\bm P_{\bm W} \bm u-\bm u\|_{\bm H^s(K)}.
	\end{align*}
	This proves \eqref{projection_error_1f}.
	
	Next, we prove \eqref{projection_error_2f}. By the same arguments we have 
	\begin{align*}
	\|{\Pi_Vq}- q\|_{ H^s(K)} \le C h_K^{-s}\|{\Pi_V q}-  q\|_{K} +C h_K^{-s}\| P_{V} q - q\|_{K} + \| P_{V} q- q\|_{ H^s(K)}.
	\end{align*} 
	By Lemma 7.2 in \cite{MR3702417} to get 
	\begin{align}\label{trace_in}
	\|\bm{\Pi_W\bm u}-\bm u\|_{\partial K}\le C\left(h_K^{-1/2}	\|\bm{\Pi_W\bm u}-\bm u\|_{K} + h_K^{s-1/2} 	\|\bm{\Pi_W\bm u}-\bm u\|_{\bm H^s(K)}  \right).
	\end{align}
	Using estimates \eqref{projection_error_2}, \eqref{projection_error_1}, \eqref{trace_in} and  \eqref{projection_error_1f}, we can  obtain \eqref{projection_error_2f}. 
\end{proof}

Since $\mathcal {P}_0(K)\in \widetilde{ V}(K)$ and $[\mathcal {P}_0(K)]^2\in \widetilde{\bm W}(K)$ {with appropriate projection error bounds (see Section~\ref{messy_discussion})},  by \Cref{main_result_error_analysis} and \Cref{projection_errorf}, we have the following corollary.
\begin{corollary}
	Let $(\Psi,\bm\Phi)\in  H({\rm curl};\Omega)\times [\bm H_0({\rm curl};\Omega)\cap\bm H({\rm div}^0;\Omega)]$ be the solution of \eqref{dual_Mixed} and  assume that the regularity result \eqref{regularity_dual} holds, then for $s\in (1/2, 1]$, we have 
	\begin{subequations}
		\begin{align}
		\|\bm{\Pi_W\bm \Phi}-\bm \Phi\|_{\mathcal T_h} +\|{\Pi_V \Psi}- \Psi\|_{\mathcal T_h}  \le C h^s \|\bm \Theta\|_{\mathcal T_h},\label{really_need1}\\
		\|\bm{\Pi_W\bm \Phi}-\bm \Phi\|_{\bm H^s(\mathcal T_h)} +\|{\Pi_V \Psi}- \Psi\|_{H^s(\mathcal T_h)}  \le C \|\bm \Theta\|_{\mathcal T_h}.\label{really_need2}
		\end{align}		
	\end{subequations}
	
\end{corollary} 
We can now prove our main result: \Cref{main_result_error_analysis}.
\subsection{Proof of \Cref{main_result_error_analysis}}
\label{proof_of_main_result_2}
First, we define the following HDG operator $\mathscr B: [V_h\times \bm W_h\times M_h]^2\to \mathbb C$
\begin{align}\label{def_maxwell_HDG_B}
\begin{split}
\mathscr B(q_h, \bm u_h, \widehat{\bm u}_h; r_h, \bm v_h, \widehat{\bm v}_h) &= 	(\mu q_h, r_h)_{\mathcal{T}_h}-(\bm u_h,\nabla\times r_h)_{\mathcal{T}_h}-\langle \bm n\times\widehat{\bm u}_h, r_h \rangle_{\partial\mathcal{T}_h}\\
&\quad +(\nabla\times q_h,\bm v_h)_{\mathcal{T}_h}
+\langle{ q}_h,\bm n\times\widehat{\bm v}_h \rangle_{\partial\mathcal{T}_h}\\
&\quad +\langle\tau\bm n\times(\bm u_h-\widehat{\bm u}_h),
\bm n\times(\bm v_h-\widehat{\bm v}_h) \rangle_{\partial\mathcal{T}_h}.
\end{split}
\end{align}
By the definition of $\mathscr B$ in \eqref{def_maxwell_HDG_B}, we can rewrite the HDG formulation of the system \eqref{Maxwell_equation_HDG_form_ori} in a compact form, as follows:

\begin{lemma} 
	The HDG method seeks  $(q_h,\bm u_h, \widehat{\bm u}_h)\in V_h\times\bm W_h\times  \bm M_h^{ g}$ such that
	\begin{align}\label{Maxwell_equation_HDG_form_comp}
	\mathscr B(q_h, \bm u_h, \widehat{\bm u}_h; r_h, \bm v_h, \widehat{\bm v}_h) - (\kappa^2\epsilon_r\bm u_h,\bm v_h)_{\mathcal{T}_h}=(\bm f,\bm v_h)_{\mathcal{T}_h},
	\end{align}
	for all $( r_h,\bm v_h,\widehat{\bm v}_h)\in V_h\times\bm W_h\times\bm M_h^{ 0}$, in which $\bm M_h^{g}$ and $\bm M_h^{0}$ are defined as
	\begin{align*}
	\bm M_h^{ g}&=\{\bm\mu\in\bm M_h:\bm n\times\bm\mu|_{\partial\Omega}=(\bm {P_M}(\bm n\times g))\times\bm n\},\\
	\bm M_h^{ 0}&=\{\bm\mu\in\bm M_h:\bm n\times\bm\mu|_{\partial\Omega}= 0\},
	\end{align*}
	where $\bm {P_M}$ denotes the $L^2$-projection from $\bm L^2(F)$ onto space $\bm M(F)$. Thus if $\bm u\in \bm L^2(F)$ then $\bm{P_M}\bm u\in \bm M(F)$ satisfies
	\begin{align}\label{def_PM}
	\langle\bm{P_Mu},\bm v_h\rangle_F&=\langle\bm u,\bm v_h\rangle_F&\forall \bm v_h\in \bm M(F).
	\end{align}
\end{lemma}

Next, we give some properties of the operator $\mathscr B$ below, the proof of the following lemma is very simple and we omit it here.
\begin{lemma}\label{symmertic_of_B}
	For any $(q_h,\bm u_h, \widehat{\bm u}_h, r_h, \bm v_h, \widehat{\bm v}_h) \in [V_h\times\bm W_h\times  \bm M_h]^2$, we have 
	\begin{align*}
	\mathscr B(q_h, \bm u_h, \widehat{\bm u}_h; r_h, -\bm v_h, -\widehat{\bm v}_h) = 	\overline{\mathscr B(r_h, \bm v_h, \widehat{\bm v}_h;q_h, -\bm u_h, -\widehat{\bm u}_h)}.
	\end{align*}
\end{lemma}

\begin{lemma}
	If for all $r_h\in V_h$,  $(q_h,\bm u_h, \widehat{\bm u}_h) \in V_h\times\bm W_h\times  \bm M_h$ satisfies
	\begin{align*}
	\mathscr B(q_h, \bm u_h, \widehat{\bm u}_h; r_h, \bm 0, \bm 0) = (G,  r_h)_{\mathcal T_h},
	\end{align*}
	where $G\in L^2(\Omega)$, then we have 
	\begin{align}\label{curl_uh_qh_inequality}
	\|\nabla \times \bm u_h\|_{\mathcal T_h} \le C\left(\|q_h\|_{\mathcal T_h} + \|{\bf h}^{-1/2} \bm n\times(\bm u_h-\widehat{\bm u}_h)\|_{\partial \mathcal T_h}+\|G\|_{\mathcal T_h}\right).
	\end{align}
\end{lemma}

\begin{proof}
	By the definition of $\mathscr B$ in \eqref{def_maxwell_HDG_B}, we have 
	\begin{align}\label{curl_uh_qh_inequality_proof_1}
	(\mu q_h, r_h)_{\mathcal{T}_h}-(\bm u_h,\nabla\times r_h)_{\mathcal{T}_h}-\langle \bm n\times\widehat{\bm u}_h, r_h \rangle_{\partial\mathcal{T}_h}= (G, \bm r_h)_{\mathcal T_h}.
	\end{align}
	We take $r_h = \nabla\times \bm u_h$ in \eqref{curl_uh_qh_inequality_proof_1} and integrate by parts to get
	\begin{align*}
	(\mu q_h, \nabla\times \bm u_h)_{\mathcal{T}_h}-(\nabla\times\bm u_h,\nabla\times \bm u_h)_{\mathcal{T}_h}-\langle \bm n\times (\bm u_h-\widehat{\bm u}_h),  \nabla\times \bm u_h \rangle_{\partial\mathcal{T}_h}=(G, \nabla\times \bm u_h)_{\mathcal T_h}.
	\end{align*}
	After apply the Cauchy-Schwartz inequality and the   local inverse inequality can get our desired result.
\end{proof}

Now, we give the proof of \Cref{main_result_error_analysis}, splitting it into three steps.
\subsubsection{Step 1: Error equations and energy arguments}
\begin{lemma}\label{error_projection_lemma} 
	Let  $(q,\bm u)\in  H({\rm curl};\Omega)\times [\bm H_0({\rm curl};\Omega)\cap\bm H({\rm div}^0;\Omega)]$ be the weak solution of \eqref{Maxwell_equation_mixed_form}, then for all  $(r_h,\bm v_h,\widehat{\bm v}_h)\in  V_h\times\bm W_h\times\bm M_h^0$,  we have 
	\begin{align}\label{error_projection}
	\mathscr B({\Pi_V} q, \bm{\Pi_W}\bm u, \bm{P_M}\bm u; r_h,\bm v_h,\widehat{\bm v}_h) = (\mu({\Pi_V} q- q), r_h)_{\mathcal{T}_h}+(\nabla\times {q}, \bm v_h)_{\mathcal{T}_h}.
	\end{align}
\end{lemma}
\begin{proof}
	By the definition of $\mathscr B$ in \eqref{def_maxwell_HDG_B} and use \eqref{M_decomposition_1}
	we have	
	\begin{align*}
	\hspace{1em}&\hspace{-1em}\mathscr B({\Pi_V} q, \bm{\Pi_W}\bm u, \bm{P_M}\bm u; r_h,\bm v_h,\widehat{\bm v}_h)\\
	& = (\mu {\Pi_V} q, r_h)_{\mathcal{T}_h}-(\bm{\Pi_W}\bm u,\nabla\times r_h)_{\mathcal{T}_h}-\langle \bm n\times \bm{P_M}\bm u, r_h \rangle_{\partial\mathcal{T}_h}\\
	&\quad +(\nabla\times \Pi_V q,\bm v_h)_{\mathcal{T}_h}
	+\langle \Pi_V q,\bm n\times\widehat{\bm v}_h \rangle_{\partial\mathcal{T}_h} \\
	&\quad +\langle\tau\bm n\times(\bm{\Pi_W}\bm u- \bm{P_M}\bm u),
	\bm n\times(\bm v_h-\widehat{\bm v}_h) \rangle_{\partial\mathcal{T}_h}\\
	&=(\mu{\Pi_V} q, r_h)_{\mathcal{T}_h}-(\bm u,\nabla\times r_h)_{\mathcal{T}_h}
	-\langle\bm n\times\bm u, r_h \rangle_{\partial\mathcal{T}_h}&\text{by }\eqref{P2}, \eqref{def_PM}\nonumber\\
	&\quad + ({\Pi_Vq},\nabla\times\bm v_h)_{\mathcal{T}_h}
	+\langle {\Pi_V} q,\bm n\times(\widehat{\bm v}_h-\bm v_h) \rangle_{\partial\mathcal{T}_h}&\text{ by }\eqref{integration_by_parts1}\\
	&\quad +\langle\tau\bm n\times(\bm{\Pi_W}\bm u- \bm u),
	\bm n\times(\bm v_h-\widehat{\bm v}_h) \rangle_{\partial\mathcal{T}_h}\\
	&=(\mu{\Pi_V} q-\nabla\times\bm u, r_h)_{\mathcal{T}_h}&\text{ by }\eqref{integration_by_parts1}\\
	&\quad +({q},\nabla\times\bm v_h)_{\mathcal{T}_h}
	+\langle {\Pi_V} q,\bm n\times(\widehat{\bm v}_h-\bm v_h) \rangle_{\partial\mathcal{T}_h}&\text{by }\eqref{P1}\\
	&\quad +\langle\tau\bm n\times(\bm{\Pi_W}\bm u- \bm u),
	\bm n\times(\bm v_h-\widehat{\bm v}_h) \rangle_{\partial\mathcal{T}_h}.
	\end{align*}	
	Since $\widehat{\bm v}_h $ is single valued on interior edges and equal to zero on boundary faces, then $\langle q,\bm n\times\widehat{\bm v}_h\rangle_{\partial\mathcal{T}_h}=0$. Moreover, by \eqref{integration_by_parts1} we have 
	\begin{align*}
	\hspace{2em}&\hspace{-2em}(q,\nabla\times\bm v_h)_{\mathcal{T}_h} + \langle {\Pi_V} q,\bm n\times(\widehat{\bm v}_h-\bm v_h) \rangle_{\partial\mathcal{T}_h}\\
	&=(\nabla\times\bm{q},\bm v_h)_{\mathcal{T}_h}
	+\langle  q- {\Pi_V} q,\bm n\times(\bm v_h-\widehat{\bm v}_h) \rangle_{\partial\mathcal{T}_h}\\
	& = (\nabla\times\bm{q},\bm v_h)_{\mathcal{T}_h}
	-\langle\tau\bm n\times(\bm{\Pi_W}\bm u- \bm u),
	\bm n\times(\bm v_h-\widehat{\bm v}_h) \rangle_{\partial\mathcal{T}_h}&\text{by }\eqref{P3}.
	\end{align*}
	This implies
	\begin{align*}
	\hspace{2em}&\hspace{-2em} \mathscr B({\Pi_V} q, \bm{\Pi_W}\bm u, \bm{P_M}\bm u; r_h,\bm v_h,\widehat{\bm v}_h)\\
	&=(\mu{\Pi_V} q-\nabla\times\bm u, r_h)_{\mathcal{T}_h}+(\nabla\times {q}, \bm v_h)_{\mathcal{T}_h}\\
	& = (\mu({\Pi_V} q- q), r_h)_{\mathcal{T}_h}+(\nabla\times {q}, \bm v_h)_{\mathcal{T}_h},&\textup{ by }\eqref{Maxwell_equation_mixed_form}
	\end{align*}	
	and completes the proof.
\end{proof}	

To simplify notation, we define
\begin{align}\label{varepsilojn}
\varepsilon_h^{ q}={\Pi_Vq}- q_h,\ \ \
\varepsilon_h^{\bm u}=\bm{\Pi_W} \bm u-\bm u_h,\ \ \
\varepsilon_h^{\widehat{\bm u}}=\bm{P_M} \bm u-\widehat{\bm u}_h.
\end{align}

We subtract \eqref{Maxwell_equation_HDG_form_comp} from \eqref{error_projection} to get the following error equations.
\begin{lemma}  
	Using the notation \eqref{varepsilojn},  for any $(r_h,\bm v_h,\widehat{\bm v}_h)\in V_h,\times\bm W_h\times\bm M_h^0$, we have
	\begin{align}\label{HDG_error_equation}
	\begin{split}
	\hspace{2em}&\hspace{-2em} \mathscr B(\varepsilon^{ q}_h,\varepsilon_h^{\bm u}, \varepsilon_h^{\widehat{\bm u}},  r_h,\bm v_h,\widehat{\bm v}_h)-(\kappa^2\epsilon_r\varepsilon_h^{\bm u},\bm v_h)_{\mathcal{T}_h}\\
	& = (\mu({\Pi_V} q- q), r_h)_{\mathcal{T}_h}+(\kappa^2\epsilon_r(\bm u-\bm{\Pi_W} \bm u),\bm v_h)_{\mathcal{T}_h}.
	\end{split}
	\end{align}
\end{lemma}	
\begin{proof}
	By the definition of $\mathscr B$ in \eqref{def_maxwell_HDG_B} and \Cref{error_projection_lemma}, we get
	\begin{align*}
	\hspace{2em}&\hspace{-2em} \mathscr B(\varepsilon^{ q}_h,\varepsilon_h^{\bm u}, \varepsilon_h^{\widehat{\bm u}},  r_h,\bm v_h,\widehat{\bm v}_h)-(\kappa^2\epsilon_r\varepsilon_h^{\bm u},\bm v_h)_{\mathcal{T}_h}\\
	& = \mathscr B(\Pi_V q,\bm{\Pi_W}, \bm {P_M u},  r_h,\bm v_h,\widehat{\bm v}_h) -\mathscr B({ q}_h, {\bm u}_h, \widehat{\bm u}_h,  r_h,\bm v_h,\widehat{\bm v}_h)\\
	&\quad -  (\kappa^2\epsilon_r\bm{\Pi_W u},\bm v_h)_{\mathcal{T}_h} +  (\kappa^2\epsilon_r\bm u_h,\bm v_h)_{\mathcal{T}_h}\\
	& = \mathscr B(\Pi_V q,\bm{\Pi_W}, \bm {P_M u},  r_h,\bm v_h,\widehat{\bm v}_h) -  (\kappa^2\epsilon_r\bm{\Pi_W u},\bm v_h)_{\mathcal{T}_h}\\
	&\quad -\left[\mathscr B({ q}_h, {\bm u}_h, \widehat{\bm u}_h,  r_h,\bm v_h,\widehat{\bm v}_h)-  (\kappa^2\epsilon_r\bm u_h,\bm v_h)_{\mathcal{T}_h}\right]\\
	& = (\mu({\Pi_V} q- q), r_h)_{\mathcal{T}_h}+(\nabla\times {q}, \bm v_h)_{\mathcal{T}_h}-  (\kappa^2\epsilon_r\bm{\Pi_W u},\bm v_h)_{\mathcal{T}_h}-(\bm f, \bm v_h)_{\mathcal{T}_h}\\
	& = (\mu({\Pi_V} q- q), r_h)_{\mathcal{T}_h}+(\kappa^2\epsilon_r(\bm u-\bm{\Pi_W} \bm u),\bm v_h)_{\mathcal{T}_h},
	\end{align*}
	where we used \eqref{Mixed-2} in the last inequality.

\end{proof}

\begin{lemma}\label{energy_of_q}
	Using definition \eqref{varepsilojn},  we have the error estimate
	\begin{align}\label{es0}
	\|\sqrt{\mu}\varepsilon_h^{ q}\|_{\mathcal{T}_h}+\|\sqrt{\tau}\bm n\times(\varepsilon_h^{\bm u}-\varepsilon_h^{\widehat{\bm u}})\|_{\partial\mathcal{T}_h}
	\le C\left( \| q-{\Pi_Vq}\|_{\mathcal{T}_h}+\|\bm u-\bm{\Pi_Wu}\|_{\mathcal{T}_h}
	+\|\varepsilon_h^{\bm u}\|_{\mathcal{T}_h}\right).
	\end{align}
\end{lemma}
\begin{proof} 
	First, we take $ (r_h,\bm v_h,\widehat{\bm v}_h)=({\varepsilon_h^{ q}},\bm 0, \bm 0) $ in \eqref{HDG_error_equation} to get
	\begin{align}\label{energy_of_q_proof_1}
	\mathscr B(\varepsilon^{ q}_h,\varepsilon_h^{\bm u}, \varepsilon_h^{\widehat{\bm u}};  {\varepsilon_h^{ q}},\bm 0,\bm 0)=(\mu({\Pi_V} q- q), \varepsilon_h^{ q})_{\mathcal{T}_h}.
	\end{align}
	Next, we take $ (r_h,\bm v_h,\widehat{\bm v}_h)=(0,\varepsilon_h^{\bm u}, \varepsilon_h^{\widehat{\bm u}}) $ in \eqref{HDG_error_equation} to get
	\begin{align}\label{energy_of_q_proof_2}
	\mathscr B(\varepsilon^{ q}_h,\varepsilon_h^{\bm u}, \varepsilon_h^{\widehat{\bm u}};  0,\varepsilon_h^{\bm u}, \varepsilon_h^{\widehat{\bm u}})-(\kappa^2\epsilon_r\varepsilon_h^{\bm u},{\varepsilon_h^{\bm u}})_{\mathcal{T}_h}=(\kappa^2\epsilon_r(\bm u-\bm{\Pi_W} \bm u),{\varepsilon_h^{\bm u}})_{\mathcal{T}_h}.
	\end{align}
	By   the equations \eqref{energy_of_q_proof_1} and \eqref{energy_of_q_proof_2}, we get
	\begin{align}\label{energy_of_q_proof_3}
	\begin{split}
	\hspace{2em}&\hspace{-2em} \overline{\mathscr B(\varepsilon^{ q}_h,\varepsilon_h^{\bm u}, \varepsilon_h^{\widehat{\bm u}};  {\varepsilon_h^{ q}},\bm 0,\bm 0)} + \mathscr B(\varepsilon^{ q}_h,\varepsilon_h^{\bm u}, \varepsilon_h^{\widehat{\bm u}};  0,\varepsilon_h^{\bm u}, \varepsilon_h^{\widehat{\bm u}})-(\kappa^2\epsilon_r\varepsilon_h^{\bm u},{\varepsilon_h^{\bm u}})_{\mathcal{T}_h}\\
	&=\overline{(\mu({\Pi_V} q- q), \varepsilon_h^{ q})_{\mathcal{T}_h}}+(\kappa^2\epsilon_r(\bm u-\bm{\Pi_W} \bm u),{\varepsilon_h^{\bm u}})_{\mathcal{T}_h}.
	\end{split}
	\end{align}
	On the other hand, by the definition of $\mathscr B$ in \eqref{def_maxwell_HDG_B} to get 
	\begin{align}\label{energy_of_q_proof_4}
	\overline{\mathscr B(\varepsilon^{ q}_h,\varepsilon_h^{\bm u}, \varepsilon_h^{\widehat{\bm u}};  {\varepsilon_h^{ q}},\bm 0,\bm 0)} + \mathscr B(\varepsilon^{ q}_h,\varepsilon_h^{\bm u}, \varepsilon_h^{\widehat{\bm u}};  0,\varepsilon_h^{\bm u}, \varepsilon_h^{\widehat{\bm u}})=\|\sqrt{\mu} \varepsilon^{ q}_h\|_{\mathcal T_h}^2 + \|\sqrt{\tau}\bm n\times(\varepsilon_h^{\bm u}-\varepsilon_h^{\widehat{\bm u}})\|_{\partial \mathcal T_h}^2.
	\end{align}
	Hence, by the equation \eqref{energy_of_q_proof_3} and \eqref{energy_of_q_proof_4}, we have 
	\begin{align*}
	\hspace{2em}&\hspace{-2em}\|\sqrt{\mu} \varepsilon^{ q}_h\|_{\mathcal T_h}^2 + \|\sqrt{\tau}\bm n\times(\varepsilon_h^{\bm u}-\varepsilon_h^{\widehat{\bm u}})\|_{\partial \mathcal T_h}^2\\
	&=\overline{(\mu({\Pi_V} q- q), \varepsilon_h^{ q})_{\mathcal{T}_h}}+(\kappa^2\epsilon_r(\bm u-\bm{\Pi_W} \bm u),{\varepsilon_h^{\bm u}})_{\mathcal{T}_h}\\
	&\le C\|{\Pi_Vq}- q\|_{\mathcal{T}_h}\|\sqrt{\mu}\varepsilon_h^{ q}\|_{\mathcal{T}_h}+C\|\bm u-\bm{\Pi_Wu}\|_{\mathcal{T}_h}\|\varepsilon_h^{\bm u}\|_{\mathcal{T}_h}.
	\end{align*}
	Use of Young's inequality gives our desired result.	
\end{proof}

\subsubsection{Step 2: Duality argument}

Similarly to \Cref{error_projection_lemma}, we have:
\begin{lemma}\label{error_projection_lemma_dual}
	Let $(\Psi,\bm \Phi)\in  H({\rm curl};\Omega)\times [\bm H_0({\rm curl};\Omega)\cap\bm H({\rm div}^0;\Omega)]$ be the solution of \eqref{dual_Mixed}, then for all $(r_h,\bm v_h,\widehat{\bm v}_h)\in  V_h\times\bm W_h\times\bm M_h^0$,  we have 
	\begin{align*}
	\mathscr B({\Pi_V} \Psi, \bm{\Pi_W}\bm \Phi, \bm{P_M}\bm \Phi; r_h,\bm v_h,\widehat{\bm v}_h) = (\mu({\Pi_V} \Psi- \Psi), r_h)_{\mathcal{T}_h}+(\nabla\times \Psi, \bm v_h)_{\mathcal{T}_h}.
	\end{align*}
\end{lemma}

The next lemma gives a partial error estimate:
\begin{lemma}\label{estimation_vare_u_1}
	Assume  $\bm \Theta \in \bm H({\rm div};\Omega)$, $\nabla\cdot\bm\Theta=0$ and that the regularity estimate \eqref{regularity_dual} holds, then we have 
	\begin{align*}
	(\varepsilon_h^{\bm u},\bm{\Theta})_{\mathcal T_h}\le C h^{s}\left(
	\| q-{\Pi_V q}\|_{\mathcal{T}_h}+\|\bm{u}-\bm{\Pi_Wu}\|_{\mathcal{T}_h}\right)\|\bm \Theta\|_{\mathcal{T}_h}+Ch^s\|\varepsilon_h^{\bm u}\|_{\mathcal{T}_h}\|\bm \Theta\|_{\mathcal{T}_h}.
	\end{align*}	
\end{lemma}
\begin{proof} 
	
	First, we take $(r_h, \bm v_h,\widehat{\bm v}_h) = (-{\Pi_V} \Psi, \bm{\Pi_W}\bm \Phi, \bm{P_M}\bm \Phi)$ in \eqref{HDG_error_equation} we get
	\begin{align}\label{dual_argu_proof_1}
	\begin{split}
	\hspace{2em}&\hspace{-2em} \mathscr B(\varepsilon^{ q}_h,\varepsilon_h^{\bm u}, \varepsilon_h^{\widehat{\bm u}};  -{\Pi_V} \Psi, \bm{\Pi_W}\bm \Phi, \bm{P_M}\bm \Phi)-(\kappa^2\epsilon_r\varepsilon_h^{\bm u}, \bm{\Pi_W}\bm \Phi)_{\mathcal{T}_h}\\
	& = -(\mu({\Pi_V} q- q), {\Pi_V} \Psi)_{\mathcal{T}_h}+(\kappa^2\epsilon_r(\bm u-\bm{\Pi_W} \bm u), \bm{\Pi_W}\bm \Phi)_{\mathcal{T}_h}.
	\end{split}
	\end{align}
	By \Cref{symmertic_of_B} and \Cref{error_projection_lemma_dual} and using \eqref{dual_Mixed_2} we have 
	\begin{align}\label{dual_argu_proof_2}
	\begin{split}
	\hspace{2em}&\hspace{-2em} \mathscr B(\varepsilon^{ q}_h,\varepsilon_h^{\bm u}, \varepsilon_h^{\widehat{\bm u}};  -{\Pi_V} \Psi, \bm{\Pi_W}\bm \Phi, \bm{P_M}\bm \Phi)\\
	& = \overline{\mathscr B({\Pi_V} \Psi, \bm{\Pi_W}\bm \Phi, \bm{P_M}\bm \Phi; -\varepsilon^{ q}_h,\varepsilon_h^{\bm u}, \varepsilon_h^{\widehat{\bm u}})}\\
	&= -(\varepsilon^{ q}_h, \mu({\Pi_V} \Psi- \Psi))_{\mathcal{T}_h}+(\varepsilon_h^{{\bm u}}, \nabla\times \Psi)_{\mathcal{T}_h}\\
	&= -(\varepsilon^{ q}_h, \mu({\Pi_V} \Psi- \Psi))_{\mathcal{T}_h}+(\varepsilon_h^{{\bm u}}, \bm\Theta + \kappa^2\overline{\epsilon_r}\bm \Phi)_{\mathcal{T}_h}.
	\end{split}
	\end{align}
	Comparing \eqref{dual_argu_proof_1} with \eqref{dual_argu_proof_2} to get
	\begin{align*}
	(\varepsilon_h^{\bm u},\bm{\Theta})_{\mathcal T_h}&
	=(\mu\varepsilon_h^{ q},{\Pi_V} \Psi- \Psi)_{\mathcal{T}_h}
	-(\mu({\Pi_V} q- q),{\Pi_V} \Psi)_{\mathcal{T}_h}\\
	&\quad+(\kappa^2\epsilon_r(\bm u-\bm{\Pi_W}\bm u),\bm{\Pi_W}\bm\Phi)_{\mathcal{T}_h}-(\kappa^2\epsilon_r(\bm \Phi-\bm{\Pi_W}\bm \Phi),\varepsilon_h^{\bm u})_{\mathcal{T}_h}\\
	& = T_1 + T_2 + T_3 + T_4.
	\end{align*}
	Next, we estimate $\{T_i\}_{i=1}^4$ one by one. For the terms $T_1$ and $T_4$, by \eqref{really_need1} and estimate for $\varepsilon_h^q$ in \Cref{energy_of_q},  we have 
	\begin{align*}
	|T_1|&\le C h^s\|\bm \Theta\|_{\mathcal{T}_h} (\| q-{\Pi_Vq}\|_{\mathcal{T}_h}+
	\|\bm u-\bm{\Pi_Wu}\|_{\mathcal{T}_h}+\|\varepsilon_h^{\bm u}\|_{\mathcal{T}_h}),\\
	|T_4| &\le C h^s\|\bm \Theta\|_{\mathcal{T}_h} \|\varepsilon_h^{\bm u}\|_{\mathcal{T}_h}.
	\end{align*}
	For the remain two terms $T_2$ and $T_3$, since $\mathcal {P}_0(K)\in \widetilde{ V}(K)$ and $[\mathcal {P}_0(K)]^2\in \widetilde{\bm W}(K)$ {with appropriate estimates for the projection (see Section~\ref{messy_discussion})}, then by \eqref{really_need2} we have 
	\begin{align*}
	T_2&=|({\Pi_V} q- q,\mu{\Pi_V} \Psi- {\Pi}_{0}(\mu{\Pi_V} \Psi))_{\mathcal{T}_h}|\le C h^{s}\| q-{\Pi_Vq}\|_{\mathcal{T}_h}\|\Pi_V\Psi\|_{H^s(\Omega)}\\
	&\le C h^{s} \| q-{\Pi_Vq}\|_{\mathcal{T}_h}\|\bm \Theta\|_{\mathcal{T}_h},\\
	|T_3| &=|(\bm u-\bm{\Pi_W}\bm u, \kappa^2\overline{\epsilon_r}\bm{\Pi_W}\bm\Phi-\bm{\Pi}_0\left[\kappa^2\overline{\epsilon_r}\bm{\Pi_W}\bm\Phi\right])_{\mathcal{T}_h}|
	\le C h^{s}\|\bm u-\bm{\Pi_W u}\|_{\mathcal{T}_h}(\|\bm{\Pi_W}\bm\Phi\|_{\bm H^s(\Omega)})\\
	&\le C h^{s}\|\bm u-\bm{\Pi_W u}\|_{\mathcal{T}_h}\|\bm \Theta\|_{\mathcal{T}_h}.
	\end{align*}
	By the above estimates of $\{T_i\}_{i=1}^4$ we get
	\begin{align}\label{T2}
	(\varepsilon_h^{\bm u},\bm{\Theta})_{\mathcal T_h}\le C h^{s}\left(
	\| q-{\Pi_V q}\|_{\mathcal{T}_h}+\|\bm{u}-\bm{\Pi_Wu}\|_{\mathcal{T}_h}\right)\|\bm \Theta\|_{\mathcal{T}_h}+Ch^s\|\varepsilon_h^{\bm u}\|_{\mathcal{T}_h}\|\bm \Theta\|_{\mathcal{T}_h}.
	\end{align}
	This completes the proof.\end{proof}

We cannot set $\bm{\Theta} = \varepsilon_h^{\bm u}$ to get an estimate of $ \varepsilon_h^{\bm u}$ since $ \varepsilon_h^{\bm u}\notin   \bm H({\rm div};\Omega)$, hence we need to modify the analysis.

{Recall the shape-regular submesh $\mathcal T_h^\star$ defined in Section~\ref{messy_discussion}.  We} define $\bm W_h^\star = \{\bm u \in \bm L^2(\Omega): \bm u|_K\in   [{\mathcal P}_\ell(K)]^2, \forall K\in\mathcal T_h^\star\}$ and $\ell\ge1$ is some integer  such that $\bm W_h\subset\bm W_h^\star$.

Next, we recall the $\bm H({\rm curl};\Omega)$ conforming element in 2D. For any $\bm v\in \bm H({\rm curl};K)$, with $K$ being a simplex,  find $\bm{\Pi}_{K,\ell}^{\rm curl}\bm v\in \bm{\mathcal P}_{\ell}(K)$ such that
\begin{subequations}\label{pi-curl}
	\begin{align}
	\langle \bm n\times\bm{\Pi}_{K,\ell}^{\rm curl}\bm v,p_{\ell} \rangle_{E}
	&=\langle \bm n\times\bm v,p_{\ell} \rangle_{E},&\forall p_{\ell}\in \mathcal P_{\ell}(E),\\
	(\bm{\Pi}_{K,\ell}^{\rm curl}\bm v,\nabla\times p_{{\ell}-1})_K
	&=	(\bm v,\nabla\times p_{{\ell}-1})_K,&\forall p_{{\ell}-1}\in \mathcal P_{\ell-1}(K),
	\end{align}
	and, when $\ell\ge 2$
	\begin{align}
	(\bm{\Pi}_{K,\ell}^{\rm curl}\bm v,\nabla(b_K p_{{\ell}-2}))_K
	&=	(\bm v,\nabla(b_K p_{{\ell}-2}))_K,&\forall p_{{\ell}-2}\in \mathcal P_{\ell-1}(K)
	\end{align}
	for all edges $F$ of $K$, where $b_K$ is the bubble function of $K$ of order three.
\end{subequations}

Following a standard procedure in \cite[Lemma 3.2, Theorem 3.1]{MR2459075},  we have the following theorem:

\begin{theorem}\label{th45} 
	Equation \eqref{pi-curl} defines  a unique $\bm{\Pi}_{K,\ell}^{\rm curl}\bm v\in \bm{\mathcal P}_{\ell}(K)$, and the following estimate holds:
	\begin{align}
	\|\bm{\Pi}_{K,\ell}^{\rm curl}\bm v-\bm v\|_{0,K}\le Ch_{K}^m
	\|\bm v\|_{m,K},
	\end{align}
	with $\bm v\in  \bm H^m(\Omega)$, and $m\in (\frac{1}{2},\ell+1]$. We define $\bm{\Pi}_{K,\ell}^{\rm curl} = \bm{\Pi}_{h,\ell}^{\rm curl}|_K$ for all $K\in \mathcal T_h^*$, then $\bm{\Pi}_{h,\ell}^{\rm curl}\bm v\in \bm H({\rm curl};\Omega)$. In addition, $\bm n\times\bm v|_{\partial\Omega}=0$ implies that $\bm n\times \bm{\Pi}_{h,\ell}^{\rm curl}\bm v|_{\partial\Omega}=0$.
\end{theorem}
Furthermore, the previously defined interpolation operator commutes with curl.

\begin{lemma} Let ${\Pi}_{K,\ell-1}$ be the $L^2$ projection onto space $\mathcal P_{\ell-1}(K)$, then we have the commutativity property
	\begin{align}
	\nabla\times\bm{\Pi}_{K,\ell}^{\rm curl}\bm v={\Pi}_{K,\ell-1}\nabla\times\bm v.\label{443}
	\end{align}
\end{lemma}
\begin{proof}
	For any $p_{\ell-1}\in \mathcal P_{\ell-1}(K)$, we have
	\begin{align*}
	(\nabla\times\bm{\Pi}_{K,\ell}^{\rm curl}\bm v,p_{\ell})_K&=
	(\bm{\Pi}_{K,\ell}^{\rm curl}\bm v,\nabla\times p_{\ell})_K
	+\langle\bm n\times\bm{\Pi}_{K,\ell}^{\rm curl}\bm v,p_{\ell}\rangle_{\partial K}\\
	&=(\bm v,\nabla\times p_{\ell-1})_K
	+\langle\bm n\times\bm v,p_{\ell-1}\rangle_{\partial K}\\
	&=(\nabla\times\bm v,p_{\ell-1})_K.
	\end{align*}
\end{proof}

Following the same techniques in \cite[Proposition 4.5]{Houston_Maxwell_NM_2005} of 3D case, we have the following result for 2D.
\begin{lemma}[{c.f \cite[Proposition 4.5]{Houston_Maxwell_NM_2005}}]\label{approxiamtion_of_pi_conforming_curl}
	For any $\bm v_h\in \bm W_h$, there exists $\bm{\Pi}_h^{\rm curl,c}\bm v_h\in  \bm W_h^\star\cap \bm H_0({\rm curl;\Omega}) $ such that
	\begin{subequations}
		\begin{align}
		\|\bm v_h-\bm{\Pi}_h^{\rm curl,c}\bm v_h\|_{\mathcal T_h}&\le C\|{\bf h}^{1/2} [\![\bm n\times\bm v_h]\!]\|_{\mathcal E_h},\label{approxiamtion_of_pi_conforming_curl_a}\\
		\|\nabla\times(\bm v_h-\bm{\Pi}_h^{\rm curl,c}\bm v_h)\|_{\mathcal T_h}&\le C\|{\bf h}^{-1/2} [\![\bm n\times\bm v_h]\!]\|_{\mathcal E_h},\label{approxiamtion_of_pi_conforming_curl_b}
		\end{align}	
	\end{subequations}
	where $\bm W_h^\star = \{\bm u \in \bm L^2(\Omega): \bm u|_K\in   [{\mathcal P}_\ell(K)]^2, \forall K\in\mathcal T_h^\star\}$ and $\ell$ is some integer  such that $\bm W_h\subset\bm W_h^\star$.	
\end{lemma}

\begin{definition}\label{def_Pi_Wm}
	Let $Q_h^\star= H_0^1(\Omega)\cap \mathcal P_{\ell+1}(\mathcal T_h^{\star})$ be a finite element space with respect to the mesh $\mathcal T_h^\star$ (therefore, $\nabla Q_h^\star\subset \bm H_0({\rm curl};\Omega)\cap \bm W^*_h$) 
	with $\sigma_h\in H^1_0(\Omega)\cap Q_h^\star$ satisfy 
	\begin{align}\label{def_pi_w_m1}
	(\nabla\sigma_h,\nabla q_h)_{\mathcal T_h}=(\bm{\Pi}_h^{\rm curl,c}(\bm u_h-\bm{\Pi_Wu}),\nabla q_h)_{\mathcal T_h}
	\end{align}
	for all  $q_h\in H^1_0(\Omega)\cap Q_h^\star$. Then we define
	\begin{align}\label{def_pi_w_m2}
	\bm{\Pi_W}^{\rm m}(\bm u,\bm u_h)&=\bm{\Pi_W u}+\nabla \sigma_h.
	\end{align}
\end{definition}

It is easy to check the following lemma using \Cref{def_Pi_Wm} and \Cref{approxiamtion_of_pi_conforming_curl}, hence we omit the proof.
\begin{lemma}
	We have 
	\begin{subequations}
		\begin{gather}
		\nabla\times\bm{\Pi_W}^{\rm m}{(\bm u,\bm u_h)}=\nabla\times\bm{\Pi_Wu},\qquad 
		[\![\bm n\times\bm{\Pi_W}^{\rm m}{(\bm u,\bm u_h)}]\!]=[\![\bm n\times  \bm{\Pi_Wu}]\!],\label{projection_pro_pi_1}\\
		\nabla\times\bm{\Pi}_h^{\rm curl ,c}(\bm u_h-\bm{\Pi_W}^{\rm m}{(\bm u,\bm u_h)})=\nabla\times\bm{\Pi}_h^{\rm curl ,c}(\bm u_h-\bm{\Pi_Wu}),\label{projection_pro_pi_2}\\
		(\bm{\Pi}_h^{\rm curl,c}(\bm u_h-\bm{\Pi_W}^{\rm m}{(\bm u,\bm u_h)}),\nabla q_h)_{\mathcal T_h}=0,  \ \ \ \forall  q_h\in H^1_0(\Omega)\cap Q_h^*.\label{projection_pro_pi_3}
		\end{gather}
	\end{subequations}
	
\end{lemma}
In addition, we have the following estimates:
\begin{lemma}
	We have the following estimates:
	\begin{subequations}
		\begin{align}
		\hspace{2em}&\hspace{-2em}\|\nabla\times(\bm{\Pi}_{h}^{\rm curl,c}(\bm{\Pi_W}^{\rm m}{(\bm u,\bm u_h)}-\bm u_h) )\|_{\mathcal T_h}\nonumber\\
		&\le C\left( \|{\bf h}^{-1/2} (q-{\Pi_Vq})\|_{\mathcal{T}_h}+\|{\bf h}^{-1/2}(\bm u-\bm{\Pi_Wu})\|_{\mathcal{T}_h}
		+\|{\bf h}^{-1/2} \varepsilon_h^{\bm u}\|_{\mathcal{T}_h}\right),\label{es1}\\
		\hspace{2em}&\hspace{-2em}\|(\bm{\Pi_W}^{\rm m}{(\bm u,\bm u_h)}-\bm u_h)-\bm{\Pi}_{h}^{\rm curl,c}(\bm{\Pi_W}^{\rm m}{(\bm u,\bm u_h)}-\bm u_h)\|_{\mathcal T_h}\nonumber\\
		&\le C\left( \|{\bf h}^{1/2} (q-{\Pi_Vq})\|_{\mathcal{T}_h}+\|{\bf h}^{1/2}(\bm u-\bm{\Pi_Wu})\|_{\mathcal{T}_h}
		+\|{\bf h}^{1/2} \varepsilon_h^{\bm u}\|_{\mathcal{T}_h}\right).\label{es2}
		\end{align}	
	\end{subequations}
	
\end{lemma}

\begin{proof} We use definition \Cref{def_Pi_Wm} and \Cref{approxiamtion_of_pi_conforming_curl}.
	By the definition of $\bm{\Pi_W}^{\rm m}$ in \eqref{def_pi_w_m2} and the approximation property of $\bm{\Pi}_{h}^{\rm curl,c}$ in \Cref{approxiamtion_of_pi_conforming_curl} to get
	\begin{align*}
	\hspace{2em}&\hspace{-2em}\|\nabla\times(\bm{\Pi}_{h}^{\rm curl,c}(\bm{\Pi_W}^{\rm m}{(\bm u,\bm u_h)}-\bm u_h) )\|_{\mathcal T_h}\\
	&\le \|\nabla\times(\bm{\Pi}_{h}^{\rm curl,c}(\bm{\Pi_W}^{\rm m}{(\bm u,\bm u_h)}-\bm u_h) )-\nabla\times(\bm{\Pi_W}^{\rm m}{(\bm u,\bm u_h)}-\bm u_h )\|_{\mathcal T_h}\\
	&\quad +\|\nabla\times(\bm{\Pi_W}^{\rm m}{(\bm u,\bm u_h)}-\bm u_h )\|_{\mathcal T_h}\\
	&=\|\nabla\times(\bm{\Pi}_{h}^{\rm curl,c}(\bm{\Pi_W} \bm u-\bm u_h) )-\nabla\times(\bm{\Pi_W} \bm u-\bm u_h)\|_{\mathcal T_h}&\textup{ by }\eqref{projection_pro_pi_2}\\
	&\quad +\|\nabla\times(\bm{\Pi_Wu}-\bm u_h )\|_{\mathcal T_h}&\textup{ by }\eqref{projection_pro_pi_1}\\
	&\le C\|{\bf h}^{-1/2}[\![\bm n\times(\bm{\Pi_Wu}-\bm u_h )]\!]\|_{\mathcal E_h}+\|\nabla\times(\bm{\Pi_Wu}-\bm u_h )\|_{\mathcal T_h} &\textup{ by }\eqref{approxiamtion_of_pi_conforming_curl_b}\\
	&{=} C\|{\bf h}^{-1/2}[\![\bm n\times \varepsilon_h^{\bm u}]\!]\|_{\mathcal E_h}+\|\nabla\times \varepsilon_h^{\bm u}\|_{\mathcal T_h}.
	\end{align*}
	Since $\varepsilon_h^{\widehat{\bm u}}$ is single valued on the interior faces and zero on the boundary, then we have 
	\begin{align*}
	\|\nabla\times(\bm{\Pi}_{h}^{\rm curl,c}(\bm{\Pi_W}^{\rm m}{(\bm u,\bm u_h)}-\bm u_h) )\|_{\mathcal T_h}\le C\|{\bf h}^{-1/2}[\![\bm n\times (\varepsilon_h^{\bm u}-\varepsilon_h^{\widehat{\bm u}})]\!]\|_{\mathcal E_h}+\|\nabla\times \varepsilon_h^{\bm u}\|_{\mathcal T_h}.
	\end{align*}
	Hence, \eqref{curl_uh_qh_inequality} and \eqref{es0} give the proof of \eqref{es1}.

	Next, by the approximation of $\bm{\Pi}_{h}^{\rm curl,c}$ in \Cref{approxiamtion_of_pi_conforming_curl} and \eqref{projection_pro_pi_1} to get
	\begin{align*}
	\hspace{2em}&\hspace{-2em}\|(\bm{\Pi_W}^{\rm m}{(\bm u,\bm u_h)}-\bm u_h)-\bm{\Pi}_{h}^{\rm curl,c}(\bm{\Pi_W}^{\rm m}{(\bm u,\bm u_h)}-\bm u_h)\|_{\mathcal T_h}\\
	&\le C\|{\bf h}^{1/2}[\![\bm n\times(\bm{\Pi_W}^{\rm m}{(\bm u,\bm u_h)}-\bm u_h)]\!]\|_{\mathcal E_h}\\
	&=   C\|{\bf h}^{1/2}[\![\bm n\times (\varepsilon_h^{\bm u}-\varepsilon_h^{\widehat{\bm u}})]\!]\|_{\mathcal E_h}.
	\end{align*}
	Finally,  \eqref{es0} gives the proof of \eqref{es2}.
\end{proof}

Next, we prove the following lemma which is  similar in \cite[Lemma 4.5]{Hiptmair_electromagnetism_Acta_2002}.
\begin{lemma}\label{theta}
	Let $\bm{\Theta}\in \bm H_0({\rm curl};\Omega)$ satisfy
	\begin{align}
	\nabla\times\bm{\Theta}=\nabla\times\bm w_h, \quad 	\nabla\cdot{\bm \Theta} =0 \quad \text{ in }\Omega,\label{448}
	\end{align}
	where $\bm w_h\in \bm H_0({\rm curl};\Omega)\cap \bm W_h^\star$ and $(\bm w_h,\nabla q_h)_{\Omega}=0$ for all $ q_h\in Q_h^\star$. Then we have 
	\begin{align}
	\|\bm w_h-\bm{\Theta}\|_{L^2(\Omega)}\le Ch^s\|\nabla\times\bm{\Theta}\|_{L^2(\Omega)},
	\end{align}
	where $s\in (\frac{1}{2},1]$ is defined in \Cref{th2}. {The following stability result also holds:
		\begin{align}\label{bound-theta}
		\|\bm{\Theta}\|_{\bm L^2(\Omega)}\le C\|\bm w_h\|_{\bm L^2(\Omega)}.
		\end{align}}
\end{lemma}
\begin{proof}  We define ${\Pi}_{\ell-1}|_K:={\Pi}_{K,\ell-1}$, then the following holds
	\begin{align}
	\nabla\times(\bm w_h-\bm{\Pi}_{h,\ell}^{\rm curl}\bm{\Theta})=
	\nabla\times\bm w_h-\Pi_{\ell-1}\nabla\times\bm{\Theta}
	=
	\nabla\times\bm w_h-\Pi_{\ell-1}\nabla\times\bm w_h
	=0.
	\end{align}
	Thus there is a $q_h\in Q_h^{\star}=\mathcal P_{\ell+1}(\mathcal T_h^\star )\cap H_0^1(\Omega)$ such that $\bm w_h-\bm{\Pi}_{h,\ell}^{\rm curl}\bm{\Theta}=\nabla q_h$.
	By a direct calculation, one can obtain
	\begin{align*}
	\|\bm w_h-\bm{\Theta}\|_{L^2(\Omega)}^2&=(\bm w_h-\bm{\Theta},\bm w_h-\bm{\Theta})_{\Omega}\\
	&=(\bm w_h-\bm{\Theta},\bm w_h-\bm{\Pi}_{h,\ell}^{\rm curl}\bm{\Theta}+\bm{\Pi}_{h,\ell}^{\rm curl}\bm{\Theta}-\bm{\Theta})_{\Omega}\\
	&=(\bm w_h-\bm{\Theta},\nabla q_h+\bm{\Pi}_{h,\ell}^{\rm curl}\bm{\Theta}-\bm{\Theta})_{\Omega}\\
	&=(\bm w_h-\bm{\Theta},\bm{\Pi}_{h,\ell}^{\rm curl}\bm{\Theta}-\bm{\Theta})_{\Omega},\end{align*}
	where we have used that $\bm w_h$ is discrete divergence free, and $\bm\Theta$ is divergence free. Now using  \Cref{th45,th2} we get	%
	\begin{align*}
	\|\bm w_h-\bm{\Theta}\|_{L^2(\Omega)}^2	%
	&\le  Ch^s\|\bm w_h-\bm{\Theta}\|_{L^2(\Omega)}\|\bm{\Theta}\|_{\bm H^s(\Omega)}\\
	&\le  Ch^s\|\bm w_h-\bm{\Theta}\|_{L^2(\Omega)}\|\nabla\times\bm{\Theta}\|_{L^2(\Omega)},
	\end{align*}
	where $s\in (\frac{1}{2},1]$ is specified in \Cref{th2}. 
	
	By the Helmoltz decomposition in two dimensions, there is a ${\phi}\in H_0^1(\Omega)$ and ${\psi}\in H^1(\Omega)$	such that
	\begin{align*}
	\bm{\Theta}=\nabla\phi+\nabla\times\psi,\qquad \|\phi\|_{H^1(\Omega)}\le \|\bm{\Theta}\|_{\bm L^2(\Omega)},\qquad \|\psi\|_{H^1(\Omega)}\le \|\bm{\Theta}\|_{\bm L^2(\Omega)}.
	\end{align*}
	Then we use the integration by parts and \eqref{448} to get
	\begin{align*}
	\|\bm{\Theta}\|_{\bm L^2(\Omega)}^2&=(\bm{\Theta},\bm{\Theta})_{\Omega}\\
	&=(\nabla\phi+\nabla\times\psi,\bm{\Theta})_{\Omega}\\
	&=-(\phi,\nabla\cdot\bm{\Theta})_{\Omega}+(\psi,\nabla\times\bm{\Theta})_{\Omega}\\
	&=(\psi,\nabla\times\bm w_h)_{\Omega}\\
	&=(\nabla\times\psi,\bm w_h)_{\Omega}\\
	&\le \|\bm{\Theta}\|_{\bm L^2(\Omega)}\|\bm w_h\|_{\bm L^2(\Omega)}.
	\end{align*}
	Thus we obtain our result.
\end{proof}

\begin{lemma}\label{main_section_error_analysis} 
	Let $(q,\bm u)\in  H({\rm curl};\Omega)\times \bm H({\rm curl};\Omega)$ and $( q_h,\bm u_h)\in V_h\times \bm W_h$ be the solution of \eqref{Maxwell_equation_mixed_form} and \eqref{Maxwell_equation_HDG_form_ori}, respectively.  Then there exists an $h_0>0$ such that for all $h\le h_0$, we have the error estimate
	\begin{align*}
	\| q- q_h\|_{\mathcal{T}_h}&\le C \left(\| q-{\Pi_V q}\|_{\mathcal{T}_h}+\|\bm u-\bm{\Pi_Wu}\|_{\mathcal{T}_h}\right),\\
	\|\bm u-\bm u_h\|_{\mathcal{T}_h}&\le C \left( h^{s-1/2}\| q-{\Pi_V q}\|_{\mathcal{T}_h}+
	C\|\bm{u}-\bm{\Pi_Wu}\|_{\mathcal{T}_h}\right).
	\end{align*}
\end{lemma}

\begin{proof}
	First, let $\bm{\Theta}\in \bm H_0({\rm curl};\Omega)\cap \bm H({\rm div};\Omega)$ be the solution of 
	\begin{align*}
	\nabla\times\bm{\Theta}=\nabla\times(\bm{\Pi}_{h}^{\rm curl,c}(\bm{\Pi_W}^{\rm m}{(\bm u,\bm u_h)}-\bm u_h)), \quad 	\nabla\cdot{\bm \Theta} =0 \quad \text{ in }\Omega.
	\end{align*}
	By \Cref{theta} and \eqref{es1}, one has 
	\begin{align}\label{estimatiuon_S_1}
	\begin{split}
	\hspace{2em}&\hspace{-2em}\|\bm{\Theta}-(\bm{\Pi}_{h}^{\rm curl,c}(\bm{\Pi_W}^{\rm m}{(\bm u,\bm u_h)}-\bm u_h) )\|_{\mathcal T_h}\\
	&\le C\|{\bf h}^s\nabla\times(\bm{\Pi}_{h}^{\rm curl,c}(\bm{\Pi_W}^{\rm m}{(\bm u,\bm u_h)}-\bm u_h) )\|_{\mathcal T_h}\\
	& \le  Ch^{s-1/2}\left(\| q-{\Pi_Vq}\|_{\mathcal{T}_h}+\|\bm u-\bm{\Pi_Wu}\|_{\mathcal{T}_h}+\|\varepsilon_h^{\bm u}\|_{\mathcal{T}_h}\right).
	\end{split}
	\end{align}
	Therefore, by the triangle inequality, \eqref{bound-theta} and \Cref{def_Pi_Wm} we have 
	
	\begin{align}\label{estimatiuon_S_2}
	\begin{split}
	\|\bm{\Theta}\|_{\mathcal T_h} &\le    \|(\bm{\Pi}_{h}^{\rm curl,c}(\bm{\Pi_W}^{\rm m}{(\bm u,\bm u_h)}-\bm u_h) )\|_{\mathcal T_h}
	\le 2 \|\bm{\Pi}_{h}^{\rm curl,c}(\bm{\Pi_W} \bm u-\bm u_h)\|_{\mathcal T_h}
	\le C \|\varepsilon_h^{\bm u}\|_{\mathcal{T}_h}.
	\end{split}
	\end{align}

	Next, we rewrite $\|\varepsilon_h^{\bm u}\|_{\mathcal{T}_h}^2$ as follows:
	\begin{align*}
	\|\varepsilon_h^{\bm u}\|_{\mathcal{T}_h}^2
	&=(\varepsilon_h^{\bm u},\bm{\Pi}_{h}^{\rm curl,c}(\bm{\Pi_W}^{\rm m}{(\bm u,\bm u_h)}-\bm u_h)-\bm{\Theta})_{\mathcal T_h}
	+(\varepsilon_h^{\bm u},\bm{\Theta})_{\mathcal T_h}\\
	&\quad+(\varepsilon_h^{\bm u},(\bm{\Pi_W}^{\rm m}{(\bm u,\bm u_h)}-\bm u_h)-\bm{\Pi}_{h}^{\rm curl,c}(\bm{\Pi_W}^{\rm m}{(\bm u,\bm u_h)}-\bm u_h))_{\mathcal T_h}\\
	&\quad+(\varepsilon_h^{\bm u},\bm{\Pi_Wu}-\bm{\Pi_W}^{\rm m}(\bm u,\bm u_h))_{\mathcal T_h}\\
	&=(\varepsilon_h^{\bm u},\bm{\Pi}_{h}^{\rm curl,c}(\bm{\Pi_W}^{\rm m}{(\bm u,\bm u_h)}-\bm u_h)-\bm{\Theta})_{\mathcal T_h}
	+(\varepsilon_h^{\bm u},\bm{\Theta})_{\mathcal T_h}\\
	&\quad+(\varepsilon_h^{\bm u},(\bm{\Pi_W}^{\rm m}{(\bm u,\bm u_h)}-\bm u_h)-\bm{\Pi}_{h}^{\rm curl,c}(\bm{\Pi_W}^{\rm m}{(\bm u,\bm u_h)}-\bm u_h))_{\mathcal T_h}\\
	&\quad-(\varepsilon_h^{\bm u},\nabla\sigma_h)_{\mathcal T_h}\\
	& = S_1 + S_2 + S_3 + S_4.
	\end{align*}
	
	The first three terms $S_1$,  $S_2$ and $S_3$ have been estimated in \eqref{estimatiuon_S_1}, \Cref{estimation_vare_u_1}, \eqref{estimatiuon_S_2}, and \eqref{es2}, respectively. We next estimate the last term $S_4$ by taking  $(r_h, \bm v_h,\widehat{\bm v}_h)=(0,\nabla\sigma_h,\nabla\sigma_h-(\bm n\cdot \nabla\sigma_h)\bm n)$ in \eqref{Maxwell_equation_HDG_form_comp} to get
	\begin{align}
	(\kappa^2\epsilon_r\bm u_h,\nabla\sigma_h)_{\mathcal T_h}=-(\bm f,\nabla\sigma_h)_{\mathcal T_h}.
	\end{align}
	Moreover, we have $-(\bm f,\nabla\sigma_h)_{\mathcal T_h}=(\kappa^2\epsilon_r\bm u,\nabla\sigma_h)_{\mathcal T_h}$, since $\epsilon_r$ is  a constant therefore
	\begin{align*}
	(\bm u_h-\bm u,\nabla \sigma_h)_{\mathcal T_h}=0.
	\end{align*}
	This implies
	\begin{align*}
	|S_4|=|(\bm u_h-\bm u,\nabla\sigma_h)_{\mathcal T_h}+
	(\bm u-\bm{\Pi_Wu},\nabla\sigma_h)_{\mathcal T_h}|
	=|(\bm u-\bm{\Pi_Wu},\nabla\sigma_h)_{\mathcal T_h}|
	\le C\|\bm u-\bm{\Pi_Wu}\|_{\mathcal T_h}\|\varepsilon_h^{\bm u}\|_{\mathcal T_h}.
	\end{align*}
	By the above estimations of $\{S_i\}_{i=1}^4$,  there exists an $h_0>0$ such that for all $h\le h_0$, we have 
	\begin{align*}
	\|\varepsilon_h^{\bm u}\|_{\mathcal{T}_h}\le C \left(h^{s-1/2}\| q-{\Pi_V q}\|_{\mathcal{T}_h}+
	\|\bm{u}-\bm{\Pi_Wu}\|_{\mathcal{T}_h}\right).
	\end{align*}
	By the above estimate and \Cref{energy_of_q} we get
	\begin{align*}
	\|\varepsilon_h^{ q}\|_{\mathcal{T}_h} \le C\left( \| q-{\Pi_Vq}\|_{\mathcal{T}_h}+\|\bm u-\bm{\Pi_Wu}\|_{\mathcal{T}_h}\right).
	\end{align*} 	
	Combining the above estimates with the triangle inequality gives the desired result.
\end{proof}

\subsubsection{Step 3: Post-processing}
Let $\bm W^{\star}(K)$ be a finite element space, we first define the following space:
\begin{align}\label{def_V_star}
V^{\star}(K)=\{v: \nabla v\in  \bm W^{\star}(K)\}.
\end{align}
The post-processing method reads: we seek $\bm u_h^{\star}\in \bm W^{\star}(K)$ such that
\begin{subequations}\label{post_processing}
	\begin{align}
	(\nabla\times\bm u^{\star}_h,\nabla\times\bm w)_K&=(q_h,\nabla\times\bm w)_K,&\text{ for all }\bm w\in \bm W^{\star}(K),\label{post_processing_1}\\
	(\bm u^{\star}_h,\nabla v)_K&=(\bm u_h,\nabla v)_K,&\text{ for all }v\in V^{\star}(K).\label{post_processing_2}
	\end{align}
\end{subequations}
Now, we state the main result in this section.
\begin{lemma} 
	Let $( q,\bm u)$ be the solution of \eqref{Maxwell_equation_mixed_form}. Then the system \eqref{post_processing} is well-defined and there exists an $h_0>0$ such that for all $h\le h_0$, we have the error estimate
	\begin{align*}
	\|\nabla\times(\bm u-\bm u_h^{\star})\|_{\mathcal{T}_h} \le C \left( \| q_h - q\|_{\mathcal{T}_h}+\inf_{\bm w_h\in \bm W^{\star}(\mathcal T_h)}\|\nabla\times(\bm u-\bm w_h)\|_{\mathcal{T}_h}\right).
	\end{align*}
\end{lemma}
\begin{proof}
	Since the  constraints for \eqref{post_processing_1} and \eqref{post_processing_2}  are
	$\dim \bm W^{\star}(K)-\dim V^{\star}(K)+1$ and
	$\dim V^{\star}(K)-1$, then \eqref{post_processing} is a square system. Therefore,  we only need to prove  uniqueness for \eqref{post_processing}. Let $q_h=0$ and $\bm u_h=0$ in \eqref{post_processing} and we take  $\bm w=\bm u_h^{\star}$ in \eqref{post_processing_1} to get $\nabla\times\bm u_h^\star=0$. Next, by the definition of $V^\star (K)$ in \eqref{def_V_star}, there is a $v\in V^{\star}(K)$ such that $\nabla v= \bm u_h^{\star}$. By \eqref{post_processing_2}, we get $\bm u_h^{\star}=\bm 0$. This  proves the uniqueness.
	
	For any $\bm w_h\in \bm W^{\star}(K)$, we rewrite the system \eqref{post_processing} into
	\begin{subequations}\label{p22}
		\begin{align}
		(\nabla\times\bm u^{\star}_h-\nabla\times\bm w_h,\nabla\times\bm w)_K&=(q_h-\nabla\times\bm w_h,\nabla\times\bm w)_K&\text{ for all }\bm w\in \bm W^{\star}(K),\label{p11}\\
		(\bm u^{\star}_h-\bm w_h,\nabla v)_K&=(\bm u_h-\bm w_h,\nabla v)_K&\text{ for all }v\in V^{\star}(K).
		\end{align}
	\end{subequations}
	By \eqref{p11} te get
	\begin{align*}
	\|\nabla\times(\bm u_h^{\star}-\bm w_h)\|_{\mathcal{T}_h} \le C \|q_h-\nabla\times\bm w_h\|_{\mathcal{T}_h}\le C\left(\|q_h-q\|_{\mathcal{T}_h}+\|\nabla\times(\bm u-\bm w_h)\|_{\mathcal{T}_h}\right).
	\end{align*}
	Using the above estimate and the triangle inequality gives the desired result.
\end{proof}

In practice, problem \eqref{post_processing} is complicated to implement. The next lemma provides a simple way to do this, that is equivalent to  \eqref{post_processing}.

\begin{lemma}
	The post-processing problem \eqref{post_processing} is equivalent to the following system:
	find $(\bm u_h^{\star},\eta_h,\gamma_h )\in \bm W^{\star}(K)\times V^{\star}(K)\times \mathcal P_0(K)$, such that
	\begin{subequations}\label{post2}
		\begin{align}
		(\nabla\times\bm u^{\star}_h,\nabla\times\bm w)_K
		+(\nabla\eta_h,\bm w)_K
		&=(q_h,\nabla\times\bm w)_K&\text{ for all }\bm w\in \bm W^{\star}(K),\\
		(\bm u^{\star}_h,\nabla v)_K+(\gamma_h,v)_K&=(\bm u_h,\nabla v)_K&\text{ for all }v\in V^{\star}(K),\\
		(\eta_h,s)_K&=0&\text{ for all }s\in \mathcal P_0(K).
		\end{align}
	\end{subequations}	
\end{lemma}
\begin{proof} 
	To prove this, we only need to prove \eqref{post2} is well-defined and $\eta_h=\gamma_h=0$. It is obvious to see that the system \eqref{post2} is a square system, hence we only need to prove the uniqueness. We take $\bm w=\nabla\eta_h$, $v=\gamma_h$ and $s=1$ in \eqref{post2} to get $\nabla\eta_h=0, \gamma_h=0$ and $(\eta_h,1)=0$. Hence $\eta_h=\gamma_h=0$.
\end{proof}

\section{Construction and numerical experiments}
\label{construction_and_numerical_experiments}

In this section, we shall present some concrete examples of spaces $V(K), \bm W(K), \bm M(\partial K)$ and  associated spaces $\widetilde V(K)$ and $\widetilde {\bm  W}(K)$ that satisfy the definition of the $M$-decomposition; see \Cref{def_M_decomposition}. In addition the spaces $\widetilde V(K)$ and $\widetilde {\bm  W}(K)$ need to satisfy (\ref{approx-prop}).  The approach to constructing the upcoming spaces follows \cite{Cockburn_M_decomposition_Part2_M2AN_2017}.  It is straightforward to check that
the examples in Table~\ref{examples} have an $M$-decomposition.  In this section we consider higher order families of elements.

Based on the construction below,  conditions \eqref{M_decomposition_1}, \eqref{M_decomposition_2} and \eqref{M_decomposition_4} are easy to check, while   condition \eqref{M_decomposition_3} is not always obvious. Fortunately, the following equivalent of condition \eqref{M_decomposition_3} is easy to check in our construction.

\begin{lemma}\label{condtion_3c_qauivalence} 
	Assume the conditions \eqref{M_decomposition_1}, \eqref{M_decomposition_2} and \eqref{M_decomposition_4} hold,  then  \eqref{M_decomposition_3} is equivalent to
	\begin{subequations}
		\begin{align}
		\gamma \text{ is  injective on the space }\widetilde V^{\perp}(K),\label{condtion_3c_qauivalence_proof2}\\
		\gamma \text{ is  injective on the space }\widetilde {\bm W}^{\perp}(K),\label{condtion_3c_qauivalence_proof3}\\
		\gamma\widetilde{V}^{\perp}(K)\oplus\gamma\widetilde{\bm W}^{\perp}(K)=\bm M(\partial K), \label{condtion_3c_qauivalence_proof4}
		\end{align}
		where $\gamma\widetilde{ V}^{\perp}(K):=\{\bm n\times v^{\perp}|_{\partial K}: v^{\perp}\in\widetilde{ V}^{\perp}(K) \}$ and 
		$\gamma\widetilde{\bm W}^{\perp}(K):=\{\bm n\times\bm w^{\perp}\times\bm n|_{\partial K}:\bm w^{\perp}\in\widetilde{\bm W}^{\perp}(K) \}$.
	\end{subequations}
	
\end{lemma}
\begin{proof}
	By \eqref{the_triple_orthogonal_property_proof_1}, it is easy to see that
	\begin{align}\label{condtion_3c_qauivalence_proof5}
	{\rm tr}:\left( \widetilde{ V}^{\perp}(K)\times\widetilde{\bm W}^{\perp}(K)\right)=	\gamma\widetilde{V}^{\perp}(K)\oplus\gamma\widetilde{\bm W}^{\perp}(K).
	\end{align}
	Assuming \eqref{condtion_3c_qauivalence_proof2}-\eqref{condtion_3c_qauivalence_proof4} hold, we need to prove that \eqref{M_decomposition_3} holds. By \eqref{condtion_3c_qauivalence_proof4} and \eqref{condtion_3c_qauivalence_proof5} it holds
	\begin{align}
	{\rm tr}:\left( \widetilde{ V}^{\perp}(K)\times\widetilde{\bm W}^{\perp}(K)\right)=\bm M(\partial K).
	\end{align}
	This proves the mapping $\rm{tr}$ is  surjective. Next, we prove $\rm{tr}$ is  injective. 
	Let $\bm q\in\bm M(\partial K)$, then there exist $v_1,v_2\in \widetilde V^{\perp}(K)$ and $\bm w_1,\bm w_2\in \widetilde{\bm W}^{\perp}(K)$ such that
	\begin{align*}
	\bm q&=\bm n\times v_1+\bm n\times\bm w_1\times\bm n,\\
	\bm q&=\bm n\times v_2+\bm n\times\bm w_2\times\bm n.
	\end{align*}
	Therefore,
	\begin{align*}
	\bm n\times(v_1-v_2)+\bm n\times(\bm w_1-\bm w_2)\times\bm n=\bm 0.
	\end{align*}
	Using  \eqref{the_triple_orthogonal_property_proof_1} we get $\bm n\times(v_1-v_2)=\bm 0$ and $\bm n\times(\bm w_1-\bm w_2)\times\bm n=0$. By \eqref{condtion_3c_qauivalence_proof3} and \eqref{condtion_3c_qauivalence_proof4}, it holds
	$v_1=v_2$ and $\bm w_1=\bm w_2$. This proves ${\rm tr}$ is  injective on the space
	$ \widetilde{ V}^{\perp}(K)\times\widetilde{\bm W}^{\perp}(K)$ and hence we finish the proof of condition \eqref{M_decomposition_3}.
	
	On the other hand, assume that condition \eqref{M_decomposition_3} holds so we need  to prove \eqref{condtion_3c_qauivalence_proof2}-\eqref{condtion_3c_qauivalence_proof4} hold. By \eqref{the_triple_orthogonal_property_3}, the condition  \eqref{condtion_3c_qauivalence_proof4} holds. Next, we prove \eqref{condtion_3c_qauivalence_proof2} and \eqref{condtion_3c_qauivalence_proof3}. Let $v_1,v_2\in \widetilde V^{\perp}(K)$ and $\bm w_1,\bm w_2$ in $\widetilde{\bm W}^{\perp}(K)$ such that
	\begin{align*}
	\bm n\times v_1&=\bm n\times v_2,\\
	\bm n\times \bm w_1\times\bm n&=\bm n\times\bm w_2\times\bm n.
	\end{align*}
	Then it holds
	\begin{align}
	\bm n\times(v_1-v_2)+\bm n\times(\bm w_1-\bm w_2)\times\bm n=\bm 0
	\end{align}
	Since ${\rm tr}$ is an injective on space
	$ \widetilde{ V}^{\perp}(K)\times\widetilde{\bm W}^{\perp}(K)$,  then it holds $v_1=v_2$ and $\bm w_1=\bm w_2$. This proves that \eqref{condtion_3c_qauivalence_proof2} and \eqref{condtion_3c_qauivalence_proof3} hold and hence completes proof.
	
\end{proof}

In this section, we shall show some numerical experiments for each choice of element. In all numerical experiments, we  take $ \mu=1,$ and $\kappa^2\epsilon_r=10$. The exact solution is
\begin{align*}
u_1=\sin(2\pi x)\sin(2\pi y),\quad
u_2=\sin( \pi x)\sin( \pi y),\quad
q  =\pi\cos(\pi x)\sin(\pi y)-2\pi\sin(2\pi x)\cos(2\pi y).
\end{align*}
Boundary data is chosen so that the above functions satisfy (\ref{Maxwell_equation_ori_form})

The post-processing spaces in all experiments are taken as 
\begin{align*}
V^{\star}(K)=\mathcal P_{k+2}(K),\qquad \bm W^{\star}(K)=\bm{\mathcal P}_{k+1}(K).
\end{align*}

\subsection{Triangle Mesh}

{We assume that the mesh $\mathcal T_h$ consists of shape regular triangles and choose $\mathcal T_h^*=\mathcal T_h$ (see Section~\ref{messy_discussion}).}
We might hope that standard $\mathcal P_k$ polynomial spaces could work, and this is indeed the case as shown in the next lemma.
\begin{lemma}\label{example_triangle_M_S_index}
	For any integer $k\ge 1$, let 
	\begin{align*}
	&V(K)=\mathcal P_k(K),\qquad \bm W(K)=\bm{\mathcal P}_k(K), \\
	&\bm M(\partial K)=\{\bm \mu: \bm\mu|_F= \bm n\times p_k,\mbox{ for some }p_k\in \mathcal{P}_k(F)\mbox{ and  for each edge }F\subset\partial K\},
	\end{align*}
	then we have
	\begin{align*}
	I_M(V(K)\times\bm W(K))=0.
	\end{align*}
\end{lemma}
\begin{proof}
	It is easy to see that 
	\begin{align*}
	\dim \{\bm n\times v|_{\partial K}:v\in V(K),\nabla \times v=\bm 0\}=1,\ \ \dim \bm M(\partial K)=3(k+1).
	\end{align*}
	Next, we give more details to compute
	\begin{align*}
	\hspace{2em}&\hspace{-6em}\dim\{\bm n\times\bm w\times \bm n|_{\partial K}:\bm w\in \bm W(K),\nabla\times\bm w=0\}\\
	&=\dim\{\bm n\times(\nabla p_{k+1})\times\bm n:p_{k+1}\in\mathcal P_{k+1}(K)\}\\
	&=\dim\{\nabla p_{k+1}: p_{k+1}\in\mathcal P_{k+1}(K)\}\\
	&\quad-\dim \{p_{k+1}\in \mathcal P_{k+1}(K):\bm n\times(\nabla p_{k+1})\times\bm n=\bm 0\}\\
	&=\dim\{\nabla p_{k+1}:p_{k+1}\in\mathcal P_{k+1}(K)\}\\
	&\quad-\dim \{ \lambda_1\lambda_2\lambda_3p_{k-2}:p_{k-2}\in \mathcal P_{k-2}(K)\}\\
	&=\binom{k+3}{2}-1-\binom{k}{2}.
	\end{align*}
	By the definitions of $I_M$ in \eqref{index_of_I_M} to get
	\begin{align*}
	I_M(V(K)\times\bm W(K))=0.
	\end{align*}
\end{proof}

By \Cref{index_IM_M_decomposition} and \Cref{example_triangle_M_S_index}, we know the finite element spaces $	\mathcal P_k(K)\times \bm{\mathcal P}_k(K)$ admits an $M$-decomposition, i.e., there exist two spaces $\widetilde{V}(K)$ and $\widetilde{\bm W}(K)$ satisfy the \Cref{def_M_decomposition}. In the following \Cref{Simple_V_W_tilde_triangle}, we give a concrete construction of the spaces $\widetilde{V}(K)$ and $\widetilde{\bm W}(K)$.

\begin{lemma}\label{Simple_V_W_tilde_triangle}
	For any integer $k\ge 1$, let 
	\begin{gather*}
	V(K)=\mathcal P_k(K),\qquad \bm W(K)=\bm{\mathcal P}_k(K), \qquad \bm M(\partial K)=\{\bm \mu\;|\; \bm\mu|_F= \bm n\times\mathcal{P}_k(F)\mbox{ for each face }F\subset\partial K\},\\
	\widetilde{V}(K)=\mathcal P_{k-1}(K),\qquad  \widetilde{\bm W}(K)=\bm{\mathcal P}_{k-1}(K).
	\end{gather*}
	Then $V(K)$ and $\bm W(K)$ admit an $M$-decomposition with respect the space $\widetilde{V}(K)$ and $\widetilde{\bm W}(K)$.  
\end{lemma}
\begin{proof}
	We notice that \eqref{M_decomposition_1}, \eqref{M_decomposition_2} and \eqref{M_decomposition_4}  obviously hold. We only have to prove \eqref{M_decomposition_3}.  By \Cref{condtion_3c_qauivalence}, we need to check the following three conditions:
	
	\begin{itemize}
		\item[(1)] $\gamma$ is  injective on the space $\widetilde V^{\perp}(K)$:  Let $v^{\perp}\in \widetilde V^{\perp}(K)$, if $\bm n\times v^{\perp}|_{\partial K}=\bm 0$, we have $v^{\perp}|_{\partial K}=0$. If $k\le 2$, then $v^\perp|_K=0$. If $k\ge 3$, we can set $v^{\perp}=p_{k-3}\lambda_1\lambda_2\lambda_3$, where $p_{k-3}\in \mathcal P_{k-3}(K)$. Then $(v^{\perp}, p_{k-3})_K=(p_{k-3}\lambda_1\lambda_2\lambda_3,p_{k-3})_K=0$, this gives $p_{k-3}=0$, hence $v^{\perp}=0$. This proves  $\gamma$ is an injective, also  $\gamma$ is an onjective; therefore, $\gamma$  is an isomorphism, which implies  
		\begin{align*}
		\dim \gamma \widetilde V^{\perp}(K)=\dim \widetilde V^{\perp}(K).
		\end{align*}

		\item[(2)] $\gamma$ is  injective on the space $\widetilde {\bm W}^{\perp}(K)$:  Let $\bm w^{\perp}\in \widetilde{\bm W}^{\perp}(K)$ satisfy $\bm n\times\bm w^{\perp}\times\bm n=\bm 0$, this implies $\bm n\times\bm w^{\perp}=0$.  For all $v\in V(K)$, we have  $\nabla\times v\in \nabla\times V(K)\subset\widetilde{\bm W}(K)$, then
		\begin{align*}
		0=(\bm w^{\perp},\nabla\times v)_K
		=(\nabla\times\bm w^{\perp},v)_K
		+\langle\bm w^{\perp},\bm n\times v \rangle_{\partial K}=(\nabla\times\bm w^{\perp},v)_K.
		\end{align*}
		Since $\nabla\times\bm W(K)\subset V(K)$, we take $v=\nabla\times\bm w^{\perp}$ in the above equation to get $\nabla\times\bm w^{\perp}=0$, then there is a $p_{k-2}\in \mathcal P_{k-2}(K)$ such that
		\begin{align*}
		\bm w^{\perp}=\nabla (\lambda_1\lambda_2\lambda_3 p_{k-2}).
		\end{align*}
		For any $\bm f_{k-1}\in\bm{\mathcal{P}}_{k-1}(K)=\widetilde{\bm W}(K)$,
		we have
		\begin{align*}
		0=(\nabla (p_{k-2}\lambda_1\lambda_2\lambda_3),\bm f_{k-1})_K
		=(p_{k-2}\lambda_1\lambda_2\lambda_3,\nabla\cdot\bm f_{k-1})_K.
		\end{align*}
		Since $\nabla\cdot\bm{\mathcal P}_{k-1}(K)=\mathcal P_{k-2}(K)$, we can take $\bm f_{k-1}$ such that $\nabla\cdot\bm f_{k-1}=p_{k-2}$, then we have $p_{k-2}=0$, and therefore, $\bm w^{\perp}=\bm 0$, hence $\gamma$ is an injective; thus $\gamma$ is an isomorphism, which implies 
		\begin{align*}
		\dim \gamma \widetilde{\bm W}^{\perp}(K)=\dim \widetilde{\bm W}^{\perp}(K).
		\end{align*}

		\item[(3)]$\gamma\widetilde{V}^{\perp}(K)\oplus\gamma\widetilde{\bm W}^{\perp}(K)=\bm M(\partial K)$:
		on the one hand $\gamma\widetilde{V}^{\perp}(K)\oplus\gamma\widetilde{\bm W}^{\perp}(K)\subset\bm M(\partial K)$; on the other hand
		\begin{align*}
		\hspace{2em}&\hspace{-2em}\dim\bm M(\partial K) - \dim \gamma\widetilde{V}^{\perp}(K) - \dim \gamma\widetilde{\bm W}^{\perp}(K) \\
		&= \dim\bm M(\partial K) - \dim \widetilde{V}^{\perp}(K) - \dim \widetilde{\bm W}^{\perp}(K) \\
		& = \dim\bm M(\partial K) - (\dim V(K) - \dim \widetilde V(K)) -( \dim \bm W(K) - \dim {\widetilde {\bm W}}(K))\\
		& =  0.
		\end{align*}
		Hence $\gamma\widetilde{V}^{\perp}(K)\oplus\gamma\widetilde{\bm W}^{\perp}(K)=\bm M(\partial K)$  holds.
	\end{itemize}
\end{proof}

Next,  we give the canonical construction on a triangle element.
\begin{lemma}[Canonical Construction]\label{Canonical_construction} 
	For any integer $k\ge 0$, let 
	\begin{align}\label{711}
	&V(K)=\mathcal P_k(K),\qquad \bm W(K)=\bm{\mathcal P}_k(K),\nonumber \\
	&\bm M(\partial K)=\{\bm \mu: \bm\mu|_F= \bm n\times p_k,\mbox{ for some }p_k\in \mathcal{P}_k(F)\mbox{ and  for each edge }F\subset\partial K\},
	\end{align}
	and
	\begin{align}  \label{712}
	\widetilde{V}(K)=\mathcal P_{k-1}(K),\qquad  \widetilde{\bm W}(K)=\nabla\times V(K)\oplus\bm W_0(K),
	\end{align}
	where
	\begin{align*}
	\bm W_0(K)=\{\bm w_0\in \bm W(K):\nabla\times\bm w_0=0,\bm n\times\bm w_0\times \bm n=\bm 0\}.
	\end{align*}
	Then $V(K)$ and $\bm W(K)$ admit the canonical $M$-decomposition with respect the spaces $\widetilde{V}(K)$ and $\widetilde{\bm W}(K)$.  
\end{lemma}
\begin{proof} 
	Similar with \Cref{Simple_V_W_tilde_triangle}, we only need to check the three conditions in \Cref{condtion_3c_qauivalence}. Moreover, the condition \eqref{condtion_3c_qauivalence_proof2} is the same with \Cref{Simple_V_W_tilde_triangle}. Hence we check the conditions \eqref{condtion_3c_qauivalence_proof3} and \eqref{condtion_3c_qauivalence_proof4} in the following.
	\begin{itemize}
		
		\item[(1)] $\gamma$ is  injective on the space $\widetilde {\bm W}^{\perp}(K)$: Let $\bm w^{\perp}\in \widetilde{\bm W}^{\perp}(K)$ satisfy $\bm n\times\bm w^{\perp}\times\bm n=\bm 0$, by the proof in (2) in \Cref{Simple_V_W_tilde_triangle}, we get $\nabla\times\bm w^{\perp}=0$. Therefore, $\bm w^{\perp}\in \bm W_0(K)\subset \widetilde{\bm W}(K)$. This implies $\bm w^{\perp}\in \widetilde{\bm W}^{\perp}(K)\cap \widetilde{\bm W}(K)=\{0\}$, hence $\gamma$ is an injective; thus $\gamma$ is an isomorphism, which implies 
		\begin{align*}
		\dim \gamma \widetilde{\bm W}^{\perp}(K)=\dim \widetilde{\bm W}^{\perp}(K).
		\end{align*}
		
		\item[(2)]$\gamma\widetilde{V}^{\perp}(K)\oplus\gamma\widetilde{\bm W}^{\perp}(K)=\bm M(\partial K)$:
		Since $\gamma\widetilde{V}^{\perp}(K)\oplus\gamma\widetilde{\bm W}^{\perp}(K)\subset\bm M(\partial K)$, then we only need  to prove
		\begin{align*}
		\dim\bm M(\partial K) &= \dim \gamma\widetilde{V}^{\perp}(K) + \dim \gamma\widetilde{\bm W}^{\perp}(K) \\
		&= \dim \widetilde{V}^{\perp}(K) + \dim \widetilde{\bm W}^{\perp}(K) \\
		& = \dim V(K) - \dim \widetilde V(K) + \dim \bm W(K) - \dim {\widetilde {\bm W}}(K)\\
		& =  \dim V(K) - \dim \widetilde V(K) + \dim \bm W(K) - \dim(\nabla\times V(K))-\dim(\bm W_0(K)),
		\end{align*}
		i.e., we need to prove $\dim(\bm W_0(K)) = \binom k 2 $.
		
		\hspace{1em} For any $\bm w_0\in \bm W_0(K)$, since $\nabla\times\bm w_0=0$, there exists $q_{k+1}\in \mathcal{P}_{k+1}(K)$ such that
		$\bm w_0=\nabla q_{k+1}$. Therefore, 
		\begin{align*}
		\bm 0=\bm n\times \bm w_0|_{\partial K}\times\bm n=\bm n\times \nabla q_{k+1}|_{\partial K}\times\bm n.
		\end{align*}
		Therefore, there is a constant $C_0$ and $p_{k-2}\in \mathcal{P}_{k-2}(K)$ such that $q_{k+1}=p_{k-2}\lambda_1\lambda_2\lambda_3+C_0$, i.e., $\dim \bm W_0(K)=\binom{k}{2}$.
	\end{itemize}
\end{proof}

In \Cref{table1}, we show numerical results on the unit square with a uniform triangular mesh. We obtain an optimal convergence rate for the solution $\bm u$ and superconvergence rate for $\nabla \times \bm u$.
\begin{table}[H]
	\small
	\caption{Results for a uniform triangular mesh and degree $k$ elements on the unit square} $\Omega = (0,1)\times(0,1)$
	\centering
	\label{table1}
	
	\begin{tabular}{c|c|c|c|c|c|c|c|c|c|c|c}
		\Xhline{1pt}
		\multirow{2}{*}{$k$} &
		\multirow{2}{*}{$\frac{\sqrt{2}}{h}$} &
		\multicolumn{2}{c|}{$\|\bm u-\bm u_h\|_{\mathcal{T}_h}$} &
		\multicolumn{2}{c|}{$\|\nabla\times(\bm u-\bm u_h)\|_{\mathcal{T}_h}$} &
		\multicolumn{2}{c|}{$\|q-q_h\|_{\mathcal{T}_h}$}&
		\multicolumn{2}{c|}{$\|\bm u-\bm u_h^{\star}\|_{\mathcal{T}_h}$}&
		\multicolumn{2}{c}{$\|\nabla\times(\bm u-\bm u_h^{\star})\|_{\mathcal{T}_h}$} \\
		\cline{3-12}		
		& &Error &Rate  &Error &Rate  &Error &Rate&Error &Rate
		&Error &Rate \\
		\hline

		&$2^3$	&2.14e-1	&2.96 	&8.07e+0	&1.71 	&1.71e-1	&2.73 	&1.80e-1	&3.04 	&1.71e-1	& \\
		1
		&$2^4$	&4.43e-2	&2.27 	&3.53e+0	&1.19 	&3.58e-2	&2.26 	&3.64e-2	&2.31 	&3.58e-2	&2.26 \\
		&$2^5$	&1.03e-2	&2.10 	&1.67e+0	&1.08 	&8.46e-3	&2.08 	&8.44e-3	&2.11 	&8.46e-3	&2.08 \\
		&$2^6$	&2.51e-3	&2.04 	&8.10e-1	&1.04 	&2.09e-3	&2.02 	&2.04e-3	&2.05 	&2.09e-3	&2.02 \\
		&$2^7$	&6.18e-4	&2.02 	&4.00e-1	&1.02 	&5.20e-4	&2.00 	&5.03e-4	&2.02 	&5.20e-4	&2.00 \\
		
		\Xhline{1pt}
		
		&$2^3$	&2.20e-2	&3.31 	&1.47e+0	&2.25 	&1.46e-2	&3.17 	&1.80e-2	&3.34 	&1.46e-2	& \\
		2
		&$2^4$	&2.51e-3	&3.13 	&3.40e-1	&2.12 	&1.81e-3	&3.02 	&2.04e-3	&3.14 	&1.81e-3	&3.02 \\
		&$2^5$	&3.00e-4	&3.06 	&8.15e-2	&2.06 	&2.26e-4	&3.00 	&2.44e-4	&3.07 	&2.26e-4	&3.00 \\
		&$2^6$	&3.67e-5	&3.03 	&2.00e-2	&2.03 	&2.82e-5	&3.00 	&2.98e-5	&3.03 	&2.82e-5	&3.00 \\
		&$2^7$	&4.54e-6	&3.02 	&4.94e-3	&2.02 	&3.52e-6	&3.00 	&3.68e-6	&3.02 	&3.52e-6	&3.00\\ 
		
		\Xhline{1pt}	
	\end{tabular}	
\end{table}

\subsection{Parallelogram Mesh}
{The mesh $\mathcal T_h$ is assumed to consist of parallelograms. For this mesh we construct $\mathcal T_h^*$ by subdividing each parallelogram into two subtriangles. The triangular mesh is assumed to be shape regular so satisfying the requirements from Section~\ref{messy_discussion}.}

For the parallelogram mesh we have the following which shows that $\mathcal{P}_k$ is not sufficient on such elements:
\begin{lemma}\label{nogo}
	For any integer $k\ge 1$, let 
	\begin{align*}
	&V(K)=\mathcal P_k(K),\qquad \bm W(K)=\bm{\mathcal P}_k(K), \\
	&\bm M(\partial K)=\{\bm \mu: \bm\mu|_F= \bm n\times p_k,\mbox{ for some }p_k\in \mathcal{P}_k(F)\mbox{ and  for each edge }F\subset\partial K\},
	\end{align*}
	then we have
	\begin{align*}
	I_M(V(K)\times \bm W(K)) =2.
	\end{align*}
\end{lemma}
\begin{proof}
	By \Cref{example_triangle_M_S_index}, we have 
	\begin{gather*}
	\dim \{\bm n\times v|_{\partial K}:v\in V(K),\nabla \times v=0\}=1, \quad 	\dim \bm M(\partial K)=4k+4.
	\end{gather*}
	For $k\ge 1$, we have 
	\begin{align*}
	\hspace{2em}&\hspace{-2em}\dim\{\bm n\times\bm w\times \bm n|_{\partial K}:\bm w\in \bm W(K),\nabla\times \bm w=0 \}\\
	&=\dim\{\bm n\times(\nabla  p_{k+1})\times\bm n:p_{k+1}\in\mathcal P_{k+1}(K)\}\\
	&=\dim\{\nabla  p_{k+1}:p_{k+1}\in\mathcal P_{k+1}(K)\}\\
	&\quad-\dim \{p_{k+1}\in \mathcal P_{k+1}(K):\bm n\times(\nabla  p_{k+1})\times\bm n=\bm 0\}\\
	&=\dim\{\nabla  p_{k+1}:p_{k+1}\in\mathcal P_{k+1}(K)\}\\
	&\quad-\dim \{p_{k-3}\Pi_{i=1}^4(a_i\lambda_i+b_i\lambda_i+c_i) :p_{k-3}\in \mathcal P_{k-3}(K)\}\\
	&=\binom{k+3}{2}-1-\binom{k-1}{2}\\
	& = 4k+1.
	\end{align*}
	This implies our result.
\end{proof}
By enriching the space we can arrive at an $M$-decomposable set of spaces (note that since
$I_M=2$ in Lemma~\ref{nogo} we add just two functions to the spaces in that lemma):
\begin{lemma}[Enriched  Construction I]\label{pecI} For any integer $k\ge 0$, let 
	\begin{align*}
	V(K)=\mathcal P_k(K),\quad \bm W(K)=\bm{\mathcal P}_k(K)+ \nabla \span\{ x^{k+1}y,xy^{k+1}   \}, \\
	\bm M(\partial K)=\{\bm \mu: \bm\mu|_F= \bm n\times p_k,\mbox{ for some }p_k\in \mathcal{P}_k(F)\mbox{ and  for each edge }F\subset\partial K\},\\
	\widetilde{V}(K)=\mathcal P_{k-1}(K),\qquad   \widetilde{\bm W}(K)=\nabla\times V(K)\oplus\bm W_0(K).
	\end{align*}
	Then $V(K)$ and $\bm W(K)$ admit an $M$-decomposition with respect the spaces $\widetilde{V}(K)$ and $\widetilde{\bm W}(K)$.  
\end{lemma}
\begin{proof} 
	As discussed in \Cref{Simple_V_W_tilde_triangle}, we only need to prove the condition of \eqref{condtion_3c_qauivalence_proof4}. 
	Since $\gamma\widetilde{V}^{\perp}(K)\oplus\gamma\widetilde{\bm W}^{\perp}(K)\subset\bm n\times\mathcal{P}_k(\partial K)$, then we only need  to prove
	\begin{align*}
	\dim\bm M(\partial K) =  \dim V(K) - \dim \widetilde V(K) + \dim \bm W(K) - \dim(\nabla\times V(K))-\dim(\bm W_0(K)),
	\end{align*}
	i.e., we need to prove $\dim(\bm W_0(K)) = \binom {k-1} {2} $.
	
	Since $\nabla\times\bm w_0=0$, there exists $q_{k+1}\in \mathcal{P}_{k+1}(K)$, and constants $a$, $b$,
	such that
	$\bm w_0=\nabla q_{k+1}
	+a\nabla x^{k+1}y+b\nabla xy^{k+1}
	$. Therefore, by a direct calculation we can get
	\begin{align*}
	\bm 0 &=\bm n\times \bm w_0|_{\partial K}\times\bm n=\bm n\times (\nabla q_{k+1}
	+a\nabla x^{k+1}y+b\nabla xy^{k+1})|_{\partial K}\times\bm n\\
	&=\nabla_{\partial K}q_{k+1}+a\nabla_{\partial K} x^{k+1}y+b\nabla_{\partial K}xy^{k+1}.
	\end{align*}
	Then  $a=-b$ when $k=0$, $a=b=0$ when $k\ge 1$, and $q_{k+1}$ is a constant on $\partial K$. Therefore, there is a constant $C_0$ and $p_{k-3}\in \mathcal{P}_{k-3}(K)$, such that $q_{k+1}=p_{k-3}\Pi_{i=1}^4(a_ix+b_iy+c_i)+C_0$.
	Then it holds
	\begin{align*}
	\bm w_0=\nabla (p_{k-3}\Pi_{i=1}^4(a_ix+b_iy+c_i)).
	\end{align*}
	This implies
	\begin{align*}
	\dim \bm W_0(K)=\binom{k-1}{2}.
	\end{align*}
\end{proof}

Now, we give another construction:
\begin{lemma}[Enriched  Construction II] \label{pecII}
	For any integer $k\ge 0$, let 
	\begin{align*}
	&V(K)=\mathcal P_k(K),\quad \bm W(K)=\bm{\mathcal P}_k(K)+ \nabla \{ x^{k+1}y,xy^{k+1}   \}+\left(^y_{-x}\right)\widetilde{\mathcal P}_k(K), \\
	&\bm M(\partial K)=\{\bm \mu: \bm\mu|_F= \bm n\times p_k,\mbox{ for some }p_k\in \mathcal{P}_k(F)\mbox{ and  for each edge }F\subset\partial K\},\\
	&\widetilde{\bm W}(K)=\nabla\times V(K)\oplus\bm W_0(K).
	\end{align*}
	Then $V(K)$ and $\bm W(K)$ admit an $M$-decomposition with respect the spaces $\widetilde{V}(K)$ and $\widetilde{\bm W}(K)$.  
\end{lemma}

In \Cref{table2,table21}, we show  numerical results on a parallelogram with a uniform parallelogram mesh. We obtain the optimal convergence rate for the solution $\bm u$ and superconvergence rate for $\nabla \times \bm u$. In terms of accuracy, and order of convergence, the enriched space in Lemma~\ref{pecII} does not offer any advantages over the space in Lemma~\ref{pecI} as is to be expected since the space in Lemma~\ref{pecI} was already sufficiently 
enriched to have an $M$-decomposition.
\begin{table}[H]
	\small
	\caption{Results for parallelogram mesh and enriched case I on $\Omega=\{(x,y): 0\le x-\sqrt{3}y\le 1, 0\le y\le 1/2\}$}
	\centering
	\label{table2}
	
	\begin{tabular}{c|c|c|c|c|c|c|c|c|c|c|c}
		\Xhline{1pt}
		\multirow{2}{*}{$k$} &
		\multirow{2}{*}{$\frac{\sqrt{2}}{h}$} &
		\multicolumn{2}{c|}{$\|\bm u-\bm u_h\|_{\mathcal{T}_h}$} &
		\multicolumn{2}{c|}{$\|\nabla\times(\bm u-\bm u_h)\|_{\mathcal{T}_h}$} &
		\multicolumn{2}{c|}{$\|q-q_h\|_{\mathcal{T}_h}$}&
		\multicolumn{2}{c|}{$\|\bm u-\bm u_h^{\star}\|_{\mathcal{T}_h}$}&
		\multicolumn{2}{c}{$\|\nabla\times(\bm u-\bm u_h^{\star})\|_{\mathcal{T}_h}$} \\
		\cline{3-12}		
		& &Error &Rate  &Error &Rate  &Error &Rate&Error &Rate
		&Error &Rate \\
		\hline	
		
		&$2^3$	&1.30e+1	&-2.58	&1.15e+2	&-3.63	&4.35e+0	&-1.89	&1.29e+1	&-2.58	&4.35e+0	&\\
		1
		&$2^4$	&1.57e-1	&6.37	&5.11e+0	&4.49	&6.49e-2	&6.07	&1.54e-1	&6.39	&6.49e-2	&6.07\\
		&$2^5$	&1.48e-2	&3.41	&2.11e+0	&1.28	&1.08e-2	&2.59	&1.32e-2	&3.55	&1.08e-2	&2.59\\
		&$2^6$	&2.52e-3	&2.55	&1.02e+0	&1.04	&2.37e-3	&2.19	&1.93e-3	&2.77	&2.37e-3	&2.19\\
		&$2^7$	&5.23e-4	&2.27	&5.07e-1	&1.01	&5.71e-4	&2.05	&3.34e-4	&2.53	&5.71e-4	&2.05\\

		\Xhline{1pt}
		
		&$2^3$	&7.19e-1	&2.86	&9.18e+0	&2.74	&1.40e-1	&4.45	&7.15e-1	&2.84	&1.40e-1	&\\
		2
		&$2^4$	&1.46e-2	&5.62	&8.16e-1	&3.49	&2.68e-3	&5.71	&1.43e-2	&5.65	&2.68e-3	&5.71\\
		&$2^5$	&1.35e-3	&3.44	&1.66e-1	&2.29	&3.23e-4	&3.05	&1.31e-3	&3.45	&3.23e-4	&3.05\\
		&$2^6$	&1.49e-4	&3.18	&3.85e-2	&2.11	&4.03e-5	&3.00	&1.44e-4	&3.18	&4.03e-5	&3.00\\
		&$2^7$	&1.76e-5	&3.08	&9.29e-3	&2.05	&5.04e-6	&3.00	&1.70e-5	&3.09	&5.04e-6	&3.00\\

		\Xhline{1pt}	
	\end{tabular}	
\end{table}

\begin{table}[H]
	\small
	\caption{Results for parallelogram mesh and enriched  case II on $\Omega=\{(x,y): 0\le x-\sqrt{3}y\le 1, 0\le y\le 1/2\}$}
	\centering
	\label{table21}
	\begin{tabular}{c|c|c|c|c|c|c|c|c|c|c|c}
		\Xhline{1pt}
		\multirow{2}{*}{$k$} &
		\multirow{2}{*}{$\frac{\sqrt{2}}{h}$} &
		\multicolumn{2}{c|}{$\|\bm u-\bm u_h\|_{\mathcal{T}_h}$} &
		\multicolumn{2}{c|}{$\|\nabla\times(\bm u-\bm u_h)\|_{\mathcal{T}_h}$} &
		\multicolumn{2}{c|}{$\|q-q_h\|_{\mathcal{T}_h}$}&
		\multicolumn{2}{c|}{$\|\bm u-\bm u_h^{\star}\|_{\mathcal{T}_h}$}&
		\multicolumn{2}{c}{$\|\nabla\times(\bm u-\bm u_h^{\star})\|_{\mathcal{T}_h}$} \\
		\cline{3-12}		
		& &Error &Rate  &Error &Rate  &Error &Rate&Error &Rate
		&Error &Rate \\
		\hline	
		
		&$2^3$	&1.20e+1	&-0.64	&1.75e+2	&-0.54	&3.61e+0	&1.66	&1.19e+1	&-0.68	&3.61e+0	&\\
		1
		&$2^4$	&1.60e-1	&6.23	&5.18e+0	&5.08	&6.52e-2	&5.79	&1.56e-1	&6.24	&6.52e-2	&5.79\\
		&$2^5$	&1.47e-2	&3.44	&2.11e+0	&1.30	&1.08e-2	&2.59	&1.31e-2	&3.58	&1.08e-2	&2.59\\
		&$2^6$	&2.51e-3	&2.55	&1.02e+0	&1.05	&2.37e-3	&2.19	&1.91e-3	&2.78	&2.37e-3	&2.19\\
		&$2^7$	&5.22e-4	&2.27	&5.05e-1	&1.01	&5.71e-4	&2.05	&3.32e-4	&2.53	&5.71e-4	&2.05\\
		
		\Xhline{1pt}
		
		&$2^3$	&7.69e+0	&0.29	&1.35e+2	&-0.30	&1.04e+0	&1.45	&7.61e+0	&0.29	&1.04e+0	&\\
		2
		&$2^4$	&1.46e-2	&9.04	&8.17e-1	&7.37	&2.66e-3	&8.61	&1.43e-2	&9.06	&2.66e-3	&8.61\\
		&$2^5$	&1.35e-3	&3.44	&1.66e-1	&2.30	&3.22e-4	&3.04	&1.31e-3	&3.45	&3.22e-4	&3.04\\
		&$2^6$	&1.49e-4	&3.18	&3.84e-2	&2.11	&4.02e-5	&3.00	&1.44e-4	&3.18	&4.02e-5	&3.00\\
		&$2^7$	&1.76e-5	&3.08	&9.27e-3	&2.05	&5.03e-6	&3.00	&1.70e-5	&3.09	&5.03e-6	&3.00\\

		\Xhline{1pt}	
	\end{tabular}	
\end{table}

\subsection{Rectangle Mesh}
{The mesh $\mathcal T_h$ is assumed to consist of squares. For this mesh we construct $\mathcal T_h^*$ by subdividing each square into two subtriangles. The triangular mesh is  shape regular so satisfying the requirements from Section~\ref{messy_discussion} (for a general rectangular mesh, the triangular mesh must be shape regular).}

In this section, we assume that all elements $K$ are squares with edges parallel to the coordinate axes. We denote by ${\mathcal Q}_k$ the standard space of polynmials in two variables with maximum degree $k$ in each variable.  Unlike
in the parallelogram case, we consider the use of $\mathcal Q_k$ based elements as these are often used for square elements.
Our first lemma shows that simple $\mathcal Q_k$ elements alone do not suffice.
\begin{lemma}
	For any integer $k\ge 1$, let 
	\begin{align*}
	&V(K)=\mathcal Q_k(K),\qquad \bm W(K)=\bm{\mathcal Q}_k(K), \\
	&\bm M(\partial K)=\{\bm \mu: \bm\mu|_F= \bm n\times p_k,\mbox{ for some }p_k\in \mathcal{P}_k(F)\mbox{ and  for each edge }F\subset\partial K\}.
	\end{align*}
	We have
	\begin{align*}
	I_M(V(K)\times\bm W(K))=2.
	\end{align*}
\end{lemma}
\begin{proof}
	It is easy to see that
	\begin{align*}
	\dim \{\bm n\times v|_{\partial K}:v\in V(K),\nabla \times v=\bm 0\}=1, \quad  	\dim \bm M(\partial K)=4k+4.
	\end{align*}
	Moreover, we have 
	\begin{align*}
	\hspace{2em}&\hspace{-2em}\dim\{\bm n\times\bm w\times \bm n|_{\partial K}:\bm w\in \bm W(K),\nabla\times\bm w=0 \}\\
	&=\dim\{\bm n\times(\nabla p_{k+1})\times\bm n:p_{k+1}=x^{\alpha}y^{\beta},\alpha\le k,\beta\le k; \alpha =k+1, \beta=0; \alpha =0, \beta=k+1\}\\
	&=\dim\{\nabla p_{k+1}:p_{k+1}=x^{\alpha}y^{\beta}, \alpha\le k,\beta\le k\}\\
	&\quad-\dim \{p_{k+1}=x^{\alpha}y^{\beta}, \alpha\le k-2,\beta\le k-2\}\\
	&=(k+1)^2+2-1-(k-1)^2\\
	&=4k+1.
	\end{align*}
	This implies that $	I_M(V(K)\times\bm W(K))=2$ and completes our proof.
\end{proof}

The previous result shows that we must add two basis functions to the space.  A possible choice is given by the 
following lemma:

\begin{lemma}[Enriched  Construction I] 
	For any integer $k\ge 1$, let 
	\begin{gather*}
	V(K)=\mathcal Q_k(K),\quad \bm W(K)=\bm{\mathcal Q}_k(K)+ \nabla \textup{span}\{ x^{k+1}y,xy^{k+1}       \}, \\
	\bm M(\partial K)=\{\bm \mu: \bm\mu|_F= \bm n\times p_k,\mbox{ for some }p_k\in \mathcal{P}_k(F)\mbox{ and  for each edge }F\subset\partial K\},\\
	\widetilde{V}(K)=\mathcal Q_{k-1}(K),\qquad  \widetilde{\bm W}(K)=\nabla\times V(K)\oplus\bm W_0(K).
	\end{gather*}
	Then $V(K)$ and $\bm W(K)$ admit an $M$-decomposition with respect the spaces $\widetilde{V}(K)$ and $\widetilde{\bm W}(K)$.  
\end{lemma}

We omit the proofs of this and the following lemma, since they are similar to the proofs in the previous section.
An alternative choice of enriched space is given next.

\begin{lemma}[Enriched  Construction II] 
	For any integer $k\ge 0$, let 
	\begin{gather*}
	V(K)=\mathcal Q_k(K),\quad \bm W(K)=\bm{\mathcal Q}_k(K)+ \nabla  \textup{span}\{ x^{k+1}y,xy^{k+1}       \}
	+ \textup{span} \left\{ \left({ }^{x^ky^{k+1}}_{x^{k+1}y^k}\right)\right\},\\
	\bm M(\partial K)=\{\bm \mu\;|\; \bm\mu|_F= \bm n\times\mathcal{P}_k(F)\mbox{ for each edge }F\subset\partial K\}, \\ \widetilde{V}(K)=\mathcal Q_{k}(K),\qquad  \widetilde{\bm W}(K)=\nabla\times V(K)\oplus\bm W_0(K).
	\end{gather*}
	Then $V(K)$ and $\bm W(K)$ admit an $M$-decomposition with respect the spaces $\widetilde{V}(K)$ and $\widetilde{\bm W}(K)$.  
\end{lemma}

In \Cref{table3}, we show the numerical results on unit square with rectangle mesh and we obtain optimal convergence rate for the solution $\bm u$ and superconvergence rate for $\nabla \times \bm u$ using Enrichment Construction I elements.  Numerical results for Enrichment Construction II elements show that exactly the same error
is observed so we do note reproduce the results here.

\begin{table}[H]
	\small
	\caption{Results for a uniform square mesh with Enrichment case I on the unit square $\Omega = (0,1)\times(0,1)$}
	\centering
	\label{table3}

	\begin{tabular}{c|c|c|c|c|c|c|c|c|c|c|c}
		\Xhline{1pt}
		\multirow{2}{*}{$k$} &
		\multirow{2}{*}{$\frac{\sqrt{2}}{h}$} &
		\multicolumn{2}{c|}{$\|\bm u-\bm u_h\|_{\mathcal{T}_h}$} &
		\multicolumn{2}{c|}{$\|\nabla\times(\bm u-\bm u_h)\|_{\mathcal{T}_h}$} &
		\multicolumn{2}{c|}{$\|q-q_h\|_{\mathcal{T}_h}$}&
		\multicolumn{2}{c|}{$\|\bm u-\bm u_h^{\star}\|_{\mathcal{T}_h}$}&
		\multicolumn{2}{c}{$\|\nabla\times(\bm u-\bm u_h^{\star})\|_{\mathcal{T}_h}$} \\
		\cline{3-12}		
		& &Error &Rate  &Error &Rate  &Error &Rate&Error &Rate
		&Error &Rate \\
		\hline	
		
		&$2^3$	&1.78e-1	&2.60	&6.93e+0	&1.46	&1.41e-1	&2.65	&4.09e-2	&3.82	&2.17e-1	&\\
		1
		&$2^4$	&3.90e-2	&2.19	&3.05e+0	&1.19	&2.88e-2	&2.29	&5.38e-3	&2.93	&5.11e-2	&2.09\\
		&$2^5$	&9.19e-3	&2.08	&1.44e+0	&1.09	&6.83e-3	&2.08	&9.12e-4	&2.56	&1.26e-2	&2.02\\
		&$2^6$	&2.23e-3	&2.04	&6.98e-1	&1.04	&1.68e-3	&2.02	&1.96e-4	&2.22	&3.14e-3	&2.00\\
		&$2^7$	&5.51e-4	&2.02	&3.44e-1	&1.02	&4.20e-4	&2.00	&4.70e-5	&2.06	&7.85e-4	&2.00\\
		\Xhline{1pt}
		
		&$2^3$	&1.86e-2	&2.87	&8.44e-1	&2.32	&8.50e-3	&3.09	&1.37e-2	&2.08	&2.44e-2	&\\
		2
		&$2^4$	&1.10e-3	&4.08	&1.37e-1	&2.62	&8.76e-4	&3.28	&4.76e-4	&4.84	&3.09e-3	&2.98\\
		&$2^5$	&1.34e-4	&3.05	&3.31e-2	&2.05	&1.09e-4	&3.00	&5.64e-5	&3.08	&3.88e-4	&2.99\\
		&$2^6$	&1.65e-5	&3.02	&8.13e-3	&2.02	&1.36e-5	&3.00	&6.90e-6	&3.03	&4.85e-5	&3.00\\
		&$2^7$	&2.04e-6	&3.01	&2.02e-3	&2.01	&1.70e-6	&3.00	&8.53e-7	&3.02	&6.07e-6	&3.00\\	
		\Xhline{1pt}	
	\end{tabular}	
\end{table}

\section{Conclusion}
In this paper we have shown that the $M$-decomposition, together with sufficiently rich auxilary spaces, is sufficient  to guarantee optimal order convergence
for the vector 2D problem arising from Maxwell's equations.  This can be used to evaluate and construct HDG
schemes on two commonly occurring elements (triangles and squares).

An interesting extension which we have not yet completed would be to devise spaces on general quadrilateral elements that have an $M$-decomposition.  More interesting is to devise a similar theory for the full Maxwell's
equations in three dimensions.  Not only is this more complicated, but it is also essentially different compared to 2D. This will be explored in our future work. 
\bibliographystyle{amsplain}
\bibliography{HDG,Maxwell,Mypapers,ADD}

\end{document}